\documentclass[12pt,a4paper]{article}
\usepackage[T1]{fontenc}
\usepackage{cite}
\usepackage[utf8]{inputenc}
\usepackage{graphicx}
\usepackage{mathcomp}
\usepackage{multirow}
\usepackage{amsfonts}
\usepackage{amssymb}
\usepackage{amsmath}
\usepackage{amsthm}
\usepackage{amscd}
\usepackage{mathtools}
\usepackage{cancel}
\usepackage{xcolor}
\usepackage[arrow, matrix, curve]{xy}
\usepackage{enumitem}
\usepackage{tikz}
\usetikzlibrary{calc, arrows}
\usepackage{babel}
\usepackage{parskip}
\usepackage[in]{fullpage}

\newcommand{\uarc}[2]{\draw (#1,0) parabola bend ($#1*(0.5,-0.1)+#2*(0.5,0.1)$)    (#2,0);}

\makeatletter
\numberwithin{equation}{section} 
\numberwithin{figure}{section} 
 \theoremstyle{plain}    
 \newtheorem{thm}{Theorem} 
 \theoremstyle{plain}    
 \newtheorem{lem}{Lemma} 
 \theoremstyle{plain}    
 \newtheorem{cor}{Corollary} 

 \theoremstyle{remark}
 \newtheorem{rem}{Remark}
 \theoremstyle{definition}
 \newtheorem{defn}{Definition}
 \theoremstyle{definition}
 
 \theoremstyle{definition}
  \newtheorem{example}{Example}
  \theoremstyle{definition}
 
\makeatother

\newcommand{\bm}{{\mathbf m}}

\newcommand{\bor}{\mbox{ b}}

\newcommand{\be}{{\mathbf e}}

\newcommand{\bn}{{\mathbf n}}

\newcommand{\ccomp}{\operatorname{cc}}
\newcommand{\chr}{\operatorname{char}}

\usepackage{bbold}
\newcommand{\Bone}{\mathbb{1}}

\newcommand{\BZ}{\mathbb{Z}}
\newcommand{\BF}{\mathbb{F}}

\newcommand{\bpm}{\begin{pmatrix}}
\newcommand{\epm}{\end{pmatrix}}
\newcommand{\bi}{\begin{itemize}}
\newcommand{\ei}{\end{itemize}}
\newcommand{\ben}{\begin{enumerate}}
\newcommand{\een}{\end{enumerate}}

\newcommand{\rank}{\operatorname{B-rank}}

\newcommand{\KK}{\Omega}
\newcommand{\coef}{\operatorname{coef}}
\newcommand{\ind}{\operatorname{ind}}
\newcommand{\bc}{\operatorname{b}}
\begin{document}

\title{On Borel orbits of quadratic forms in characteristic $2$}
\author{ Yasmine B. Sanderson}
\date{Dec. 8, 2025}
\maketitle
\begin{center}Abstract\end{center}
{We consider the spherical variety of quadratic forms over a quadratically closed field of characteristic $2$, and determine its orbits for the action of the Borel subgroup of upper triangular matrices. We exhibit a connection between these orbits and the Catalan triangle numbers. In addition, we describe explicitly a natural Weyl group action on the set of Borel orbit double covers.}
\noindent
\section{Introduction}

Let $X = G/H$ be a homogeneous variety for a connected reductive group $G$ and let $B$ be a Borel subgroup of $G$. When $B$ has a dense orbit in $X$, then $X$ is {\it spherical} and consequently contains only finitely many $B$-orbits~\cite{brion1986quelques}, \cite{vinberg1986complexity}, \cite{knop1995set}. In characteristic $\not= 2$ spherical varieties have many nice properties. For example, there is a natural action of the Weyl group $W$ on the set of $B$-orbits \cite{richardson1990bruhat}, \cite{knop1995set}.
In characteristic $2$, however, this doesn't happen. Instead there is a Weyl group action on the set of double covers of all Borel orbits  \cite{knop1995set}. 

The point of this paper is to determine all $B$-orbits and to  make explicit this Weyl group action for a specific example. More precisely, $X$  is the space of all quadratic forms $q$ in $n$ variables $x_1, \ldots, x_n$ with coefficients in $\KK$, a quadratically closed field of characteristic $2$. We classify the $B$-orbits and show that the number of such orbits of maximal rank are expressed by Catalan and Catalan triangle numbers. We then show that the set of double covers over all such $B$-orbits is in a natural bijection with the set $M_n$ of subsequences $\bm \subset \{1, \ldots, n\}$ of length $\lfloor \frac{n}{2} \rfloor$. We show that the obvious action of $S_n$ on $M_n$ is the action described in \cite{knop1995set}.

We would like to thank the Institute for Advanced Study for its hospitality while working on this paper.

\section{Quadratic Forms and $B$-orbits}

We consider $G = GL(n,\KK)$ and the Borel subgroup $B \subset G$ consisting of upper-triangular matrices. Let $\KK$ be a quadratically closed field of characteristic $2$. Let $V_n$ be the $\KK$-vector space of quadratic forms in $n$ variables:
$$V_n = \{\sum_{1 \leq i \leq j \leq n}c_{i,j}x_ix_j \ \mid \ c_{i,j} \in \KK, \ \forall i,j\}. $$ Then $G$ acts on $V_n$. For an element $g = (g_{i,j}) \in G$ and $x_t$ we have $$g\cdot x_t = \sum_{j=1}^n g_{t,j}x_j.$$ This action is extended to $V_n$ by $g\cdot x_sx_t = (g\cdot x_s)(g\cdot x_t)$. Notice that this is a right-action since $(gh)\cdot x = h\cdot (g\cdot x)$. The action of $B$ on $V_n$ is simply the restriction of this action to $B$: for a generic element $\bor = (b_{i,j})_{1 \leq i \leq j \leq n} \in B$
\begin{equation}\bor \cdot x_sx_t =   \sum_{(s,t)\leq (i,j)}b_{s,i}b_{t,j}x_ix_j.\end{equation}\label{boreleq-1}
 where we use the usual lexicographic ordering on pairs of integers and monomials: $(i,j) < (k,l)$ (or $x_{i}x_j < x_{k}x_l$)  iff $i<k$ or $i=k$ and $j<l$. For a quadratic form $q =  \sum_{1 \leq i \leq j \leq n}c_{i,j}x_ix_j$ we set $\coef(x_{i}x_j,q) = c_{i,j}$. We let  $$\ind(q) := \{ j \ |\  x_j \mbox{ occurs in } q\} = \{ j \ | \ \exists i | c_{i j} \not= 0 \mbox{ or } c_{j i} \not= 0\ \}.$$
 We write $(k,l) \prec (i,j)$ if $(k,l) \geq (i,j)$ or if $(l,k) \geq (i,j)$, in other words when $x_kx_l $ could occur as a monomial in $\bor\cdot x_ix_j$. We now study the $B$-orbits on $V$:
 
 \begin{thm}\label{lem-borbit} {\bf 1.}  Each $B$-orbit $B\cdot q$ in $V$ contains a unique element $q_{\bn}$ in {\bf normal form}, that is, a quadratic form  $q_{\bn} = q_1 + \cdots + q_r$ with $0 \leq r $ such that
     for each $t$, $q_t = \epsilon_tx_{i_t}^2 + \delta_tx_{i_t}x_{j_t}$ where
     \begin{enumerate}
          \item $\epsilon_t, \delta_t \in \{0,1\}$ and $(\epsilon_t,\delta_t) \not= (0,0)$, i.e., $q_t \in \{ x_{i_t}^2, x_{i_t}x_{j_t}, x_{i_t}^2 +  x_{i_t}x_{j_t}\ \}$, 
            \item the sets $\{ i_t, j_t\}$ are pairwise disjoint and $j_t = i_t$ if and only if $\delta_t= 0$,
        \item $i_1 < i_2 < \cdots < i_r$ and for every $t$, $i_t \leq j_t$,
   
\item \label{lem-borbit_cond3} (C1) if $\epsilon_t=1$ and $\delta_t = 0$ then $\epsilon_s=0$ for all $s>t$ ,
  \item  \label{lem-borbit_cond4} (C2) if $s \not= t$ with $i_s < i_t < j_t < j_s$ then $\epsilon_s\epsilon_t = 0$.
  \end{enumerate}   The $q_i$ are called the {\bf normal components} of $q_{\bn}$ and every $(i_t,j_t)$ is an {\bf index pair} of $q_{\bn}$.
  
 {\bf 2.} Let $q$ be a normal quadratic form and $B_q$ the stabilizer of $q$. If $\bor \in B_q$ then $\bor = (b_{i,j})$ satisfies:
  \begin{enumerate}
  \item $b_{j_t z} = 0$ for all $z \not= j_1, j_2, \ldots, j_t$,
  \item if $\epsilon_t = 1$ then for every $s<t$ we have $\epsilon_s b_{i_s i_t} = 0$.
   
 \end{enumerate}
\end{thm}
\begin{rem}

To every normal form $q$ we can associate a diagram consisting of a row of $n$ dots (numbered $1, \ldots, n$ from left to right). Dots $s$ and $t$ are connected by an edge if and only if $(s,t)$ is an index pair for $q$. The dot $s$ is filled if and only if $x_{s}^2$ is a summand of $q$. The normal component $x_s^2$ is represented by a filled-in isolated circle. The quadratic normal form $q = x_1^2 + x_1x_5 + x_2x_3 + x_4^2+ x_4x_6 + x_7^2 + x_8x_9$ has diagram
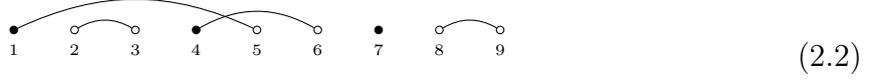
\begin{equation}\label{diag1}
  \begin{tikzpicture}[scale=0.8]
  \tikzset{vertex/.style={draw,circle,inner sep=1pt, fill=black}}
   \tikzset{overtex/.style={draw,circle,inner sep=1pt}}
    
    \node[vertex,label=below:{\tiny $1$}] (1) at (0,0) { };
    \node[overtex,label=below:{\tiny $2$}] (2) at (1,0) { };
    \node[overtex,label=below:{\tiny $3$}] (3) at (2,0) { };
    \node[vertex,label=below:{\tiny $4$}] (4) at (3,0) { };
     \node[overtex,label=below:{\tiny $5$}] (5) at (4,0) { };
    \node[overtex,label=below:{\tiny $6$}] (6) at (5,0) { };
    \node[vertex,label=below:{\tiny $7$}] (7) at (6,0) { };
    \node[overtex,label=below:{\tiny $8$}] (8) at (7,0) { };
     \node[overtex,label=below:{\tiny $9$}] (9) at (8,0) { };

    \draw (1) to [out=25,in=155] (5);
    \draw (2) to [out=30,in=150] (3);
     \draw (4) to [out=30,in=150] (6);
    \draw (8) to [out=30,in=150] (9);
  \end{tikzpicture}
\end{equation}
Notice that (C1) concerns what sort of components occur after a normal component which is a square $x_t^2$.
It says that in the diagram associated to a normal form, the following subdiagrams may occur

\begin{center}
\begin{tikzpicture}[scale=0.8]
  \tikzset{vertex/.style={draw,circle,inner sep=1pt, fill=black}}
   \tikzset{overtex/.style={draw,circle,inner sep=1pt}}
    
    \node[vertex] (1) at (0,0) { };
    \node[overtex] (2) at (1,0) { };

\end{tikzpicture} \hskip3cm \begin{tikzpicture}[scale=0.8]
  \tikzset{vertex/.style={draw,circle,inner sep=1pt, fill=black}}
   \tikzset{overtex/.style={draw,circle,inner sep=1pt}}
    
    \node[vertex] (1) at (0,0) { };
    \node[overtex] (2) at (1,0) { };
    \node[overtex] (3) at (2,0) { };
  
    \draw (2) to [out=30,in=150] (3);
\end{tikzpicture}
\end{center}
\noindent
but not a subdiagram like this:
\begin{center}
\begin{tikzpicture}[scale=0.8]
  \tikzset{vertex/.style={draw,circle,inner sep=1pt, fill=black}}
   \tikzset{overtex/.style={draw,circle,inner sep=1pt}}
    
    \node[vertex] (1) at (0,0) { };
    \node[vertex] (2) at (1,0) { };

\end{tikzpicture} \hskip3cm \begin{tikzpicture}[scale=0.8]
  \tikzset{vertex/.style={draw,circle,inner sep=1pt, fill=black}}
   \tikzset{overtex/.style={draw,circle,inner sep=1pt}}
    
    \node[vertex] (1) at (0,0) { };
    \node[vertex] (2) at (1,0) { };
    \node[overtex] (3) at (2,0) { };
  
    \draw (2) to [out=30,in=150] (3);
\end{tikzpicture}
\end{center}

Notice that (C2) concerns components associated to nested pairs of indices. It says that in the diagram associated to a normal form, the following subdiagrams may occur

\begin{center}
\begin{tikzpicture}[scale=0.8]
  \tikzset{vertex/.style={draw,circle,inner sep=1pt, fill=black}}
   \tikzset{overtex/.style={draw,circle,inner sep=1pt}}
    
    \node[vertex] (1) at (0,0) { };
    \node[overtex] (2) at (1,0) { };
    \node[overtex] (3) at (2,0) { };
    \node[overtex] (4) at (3,0) { };

    \draw (1) to [out=30,in=150] (4);
    \draw (2) to [out=30,in=150] (3);
\end{tikzpicture} \hskip2cm \begin{tikzpicture}[scale=0.8]
  \tikzset{vertex/.style={draw,circle,inner sep=1pt, fill=black}}
   \tikzset{overtex/.style={draw,circle,inner sep=1pt}}
    
    \node[overtex] (1) at (0,0) { };
    \node[vertex] (2) at (1,0) { };
    \node[overtex] (3) at (2,0) { };
    \node[overtex] (4) at (3,0) { };

    \draw (1) to [out=30,in=150] (4);
    \draw (2) to [out=30,in=150] (3);
\end{tikzpicture} \hskip2cm \begin{tikzpicture}[scale=0.8]
  \tikzset{vertex/.style={draw,circle,inner sep=1pt, fill=black}}
   \tikzset{overtex/.style={draw,circle,inner sep=1pt}}
    
    \node[overtex] (1) at (0,0) { };
    \node[overtex] (2) at (1,0) { };
    \node[overtex] (3) at (2,0) { };
    \node[overtex] (4) at (3,0) { };

    \draw (1) to [out=30,in=150] (4);
    \draw (2) to [out=30,in=150] (3);
\end{tikzpicture}    
\end{center}

\noindent
but not a subdiagram like this:
\[ \begin{tikzpicture}[scale=0.8]
  \tikzset{vertex/.style={draw,circle,inner sep=1pt, fill=black}}
   \tikzset{overtex/.style={draw,circle,inner sep=1pt}}
    
    \node[vertex] (1) at (0,0) { };
    \node[vertex] (2) at (1,0) { };
    \node[overtex] (3) at (2,0) { };
    \node[overtex] (4) at (3,0) { };

    \draw (1) to [out=30,in=150] (4);
    \draw (2) to [out=30,in=150] (3);
\end{tikzpicture}    \]
  \end{rem}
\begin{rem}\label{rem-1} If a normal quadratic form $q$ for $GL(n)$ satisfies (C1) then it contains only one pure power $x_{i_s}^2$. Set $j_s := n+1$ and let $\tilde{q} = q + x_{i_s}x_{j_s}$ be its extension to $GL(n+1)$. Then the pair $(i_s,j_s)$ satisfies $i_s < i_t < j_t < j_s= n+1$ for all $t > s$. So $q$ satisfying the (C1) condition is equivalent to $\tilde{q} $ satisfying the (C2) condition. The quadratic form $q$ is normal if and only if $\tilde{q}$ is normal.
  \end{rem}
  \begin{proof}[\bf Proof of 1. (Existence)]
    We will use a recursive method to obtain the desired quadratic form. If $q = 0$ then $B\cdot q = 0$ and $q_{\bf n} = q = 0$ (here $r=0$). So without loss of generality we can assume $q \not= 0$. 

    \noindent
   Set $i = \min (\ind(q))$. Let $j= \min \{\  t > i \ \mid \ c_{i,t} \not=0 \}$ {\it if this exists}. We have two cases to consider:

    {\bf Case $j$ doesn't exist:} That means $c_{i,i} \not= 0$ and $c_{i,t} = 0$ for all $t > i$. We use the (Borel group) operation $x_{i} \mapsto c_{i,i}^{-\frac{1}{2}}x_{i}$ (all other $x_s$ are fixed) to obtain $q = q_1'+ p_1$ where $q_1'= x_{i}^2$ and $i \not\in \ind(p_1)$.

    {\bf Case $j$ does exist:} This means that $c_{i,j} \not= 0$ and $c_{i,s}=0$ for $i < s < j$. We can express $q$ as $$q = c_{i,i}x_{i}^2 + c_{i,j}x_{i}x_{j} + x_{i}u + x_{j}v + w$$ where $i,j$ are not members of $\ind(u)$, $ \ind(v)$, $\ind(w)$.  We use the (Borel group) operation
    $$ x_{i} \mapsto \begin{cases} c_{i,i}^{-1/2}x_{i}+c_{i,j}^{-1}v & \text{ if } c_{i,i} \not= 0 \\ x_{i}+c_{i,j}^{-1}v & \text{ if } c_{i,i}= 0 \end{cases} \hskip 1cm
                   x_{j}  \mapsto \begin{cases} c_{i,j}^{-1}( c_{i,i}^{1/2}x_{j} + u)  & \text{ if } c_{i,i} \not= 0 \\c_{i,j}^{-1}(x_{j} + u) & \text{ if } c_{i,i}= 0 \end{cases} $$
                 on $q$ to obtain
                 \begin{align*} c_{i,i} c_{i,i}^{-1}(x_{i}^2 + c_{i,j}^{-2}v^2) &+ c_{i,j}( c_{i,i}^{-1/2}x_{i}+c_{i,j}^{-1}v)(c_{i,j}^{-1}( c_{i,i}^{1/2}x_{j} + u)) + \\ &+ ( c_{i,i}^{-1/2}x_{i}+c_{i,j}^{-1}v)u + c_{i,j}^{-1}( c_{i,i}^{1/2}x_{j} + u)v + w  \\ & = x_{i}^2 + x_{i}x_{j} + c_{i,j}^{-2}v^2 + c_{i,j}^{-1}uv + w \\ & = x_{i}^2 + x_{i}x_{j} + p_1\end{align*} if $c_{i,i} \not=0$ and similarly in the case $c_{i,i}=0$.
              
    We thereby obtain $q \mapsto q_1'+ p_1$ where $q_1'=  \epsilon_1x_{i}^2 +  \delta_1x_{i}x_{j} $ with $$\epsilon_1 = \begin{cases} 1 & \text{ if } c_{i,i}\not= 0 \\ 0 & \text{ otherwise }. \end{cases} \ \ \ \mbox{ and } \ \ \  \delta_1 = \begin{cases} 1 & \text{ if } c_{i, j}\not= 0 \\ 0 & \text{ otherwise } \end{cases}$$  and where $i,j \not\in \ind(p_1)$.

    We repeat the procedure with $p_1$, thereby obtaining the decomposition $ q_1'+ q_2'+ p_2$ where $\ind(q_2') \cap \ind(p_2) = \emptyset $ and $q_2'$ in the desired form. In at most $n$ steps we obtain a complete decomposition $q'  = q_1'+ \cdots + q_t'$. By construction we have $B\cdot q'  = B \cdot q$. If $q'$ satisfies (C1) and (C2) then we are done and $q_{\bf n} = q'$. If not, we use further Borel actions to obtain the normal form.

 {\bf Step 1 (satisfying (C1)):} If $q'$ satisfies (C1) then we set $q''=q'$ and move on to step 2. Otherwise, let $t:=\min \{\  l \ \mid \ q_l'= x_{i_l}^2\ \}$. By using the Borel action
  $$\bor' : x_{i_t} \mapsto x_{i_t} + \sum_{l > t}\epsilon_l x_{i_l}$$
  we eliminate all summands in $q' $  of the form $x_{i_l}^2$ ($l > t$) without affecting any of the other summands. This gives you a quadratic form $q''$ satisfying (C1). If $q''$ satisfies (C2) then we are done and  $q_{\bf n} = q''$. If not, we use yet another Borel action to obtain the normal form:

  {\bf Step 2 (satisfying (C2)):} Let $s \not= t$ be such that $i_s,j_s,i_t,j_t \in \ind(q'')$ with $i_s < i_t < j_t < j_s$ and $\epsilon_s = \delta_s = \epsilon_t = \delta_t = 1$. We use the (Borel) mapping $x_{i_s} \mapsto x_{i_s} + x_{i_t}$, $x_{j_t} \mapsto x_{j_t} + x_{j_s}$ (and all other variables fixed) to obtain
  \begin{align*}q_s'+ q_t' & = x_{i_s}^2 + x_{i_s}x_{j_s} + x_{i_t}^2 + x_{i_t}x_{j_t}\\  &\mapsto ( x_{i_s} + x_{i_t})^2 + ( x_{i_s} + x_{i_t})x_{j_s} +  x_{i_t}^2 + x_{i_t}( x_{j_t} + x_{j_s})\\ & =  x_{i_s}^2 + x_{i_s}x_{j_s} + x_{i_t}x_{j_t} =: q''_s + q''_l.\end{align*}
  In other words, components with diagrams
  \[ \begin{tikzpicture}[scale=0.8]
  \tikzset{vertex/.style={draw,circle,inner sep=1pt, fill=black}}
   \tikzset{overtex/.style={draw,circle,inner sep=1pt}}
    
    \node[vertex] (1) at (0,0) { };
    \node[vertex] (2) at (1,0) { };
    \node[overtex] (3) at (2,0) { };
    \node[overtex] (4) at (3,0) { };

    \draw (1) to [out=30,in=150] (4);
    \draw (2) to [out=30,in=150] (3);
\end{tikzpicture} \hskip 2cm \begin{tikzpicture}[scale=0.8]
  \tikzset{vertex/.style={draw,circle,inner sep=1pt, fill=black}}
   \tikzset{overtex/.style={draw,circle,inner sep=1pt}}
    
    \node[vertex] (1) at (0,0) { };
    \node[overtex] (2) at (1,0) { };
    \node[overtex] (3) at (2,0) { };
    \node[overtex] (4) at (3,0) { };

    \draw (1) to [out=30,in=150] (4);
    \draw (2) to [out=30,in=150] (3);
\end{tikzpicture}    \]
  belong to the same $B$-orbit. 
We do this for every nested pairs of indices for which (C2) is not satisfied. The resulting quadratic form $q''$ is by construction normal $q_{\bn}:= q''$. In the end, we obtain the desired quadratic normal form $q_{\bf n} \in B  q$.

\bigskip \noindent
{\bf (Uniqueness)} Let $q \not= 0$ be normal and let $p =\bor  q$ be normal for some $\bor \in B$. We have to show that $p=q$. Without loss of generality, we can replace $q$ by $ \tilde{q}$. This allows us to assume that $q$ (being normal) has the form $q = q_1 + \cdots + q_r$ where $q_s = \epsilon_sx_{i_s}^2 + x_{i_s}x_{j_s}$ where $i_1 < i_2 < \cdots < i_r$ and for all $s$ $i_s < j_s$ and $\epsilon_s \in \{0,1\}$ . 
The first part of this proof is to show that $p$ and $q$ have the same index pairs $(i_t,j_t)$, $1 \leq t \leq r$, $i_t \not= j_t$. In other words, we show that $p$ and $q$ have the same {\it mixed monomials} $x_{i_t}x_{j_t}$. Since the Borel action maps squares onto squares:
$$\bor x_i^2 = (\sum_{s \geq i} b_{i s} x_s)^2 =  \sum_{s \geq i} b_{i s}^2 x_s^2, $$
then for this section of the proof, we can, without loss of generality, ignore all pure quadratic $x_s^2$ terms. In particular we can assume that $q_s =  x_{i_s}x_{j_s}$ for all $s$.

By assumption $x_{i_1}x_{j_1}$ is the smallest mixed monomial in $ q$. Since $$\bor q = \underbrace{b_{i_1 i_1}b_{j_1 j_1}}_{\not= 0}x_{i_1}x_{j_1} + \mbox{ higher terms}, $$ we have that $x_{i_1}x_{j_1}$ is the smallest mixed monomial in $p$. So $p_1 = q_1$. The rest of the proof is by induction on $t$.

{\bf Case $t=1$:}  The normality of $p$ puts restrictions on the $\bor_{k l}$. In particular, $\coef(x_{i_1}x_l,p) = 0$ if $l \not= j_1$. Similarly $\coef(x_{k}x_{j_1},p) = 0$ if $k \not= i_1$. The $\coef(x_{i_1}x_{j_1},p)=1$ is due to normality. We have, for $y \not= i_1$
$$\coef(x_{i_1}x_y, p) =\coef(x_{i_1}x_y, b  q) = b_{i_1 i_1}b_{j_1 y} + b_{i_1 y}\underbrace{b_{j_1 i_1}}_{=0}= \begin{cases} 0 & \text{ if } y \not= j_1 \\ 1 & \text{ if } y = j_1\end{cases}.$$ So \begin{equation}\label{eq1} b_{j_1 y } = 0 \mbox{ for all }  y \not= j_1.\end{equation}

Consider now $x_yx_z$ with $y \not= z$, $y,z \not= i_1, j_1$ and $x_yx_z$ would be a monomial in $\bor x_{i_1}x_{j_1}$ but not in $\bor x_{i_2}x_{j_2}$. In other words
 $(y,z) \prec (i_1,j_1)$ but $(y,z) \not\prec (i_2,j_2)$. We have
$$ \coef(x_yx_z, p) = b_{i_1y}\underbrace{b_{j_1z}}_{=0} + b_{i_1z}\underbrace{b_{j_1z}}_{=0} = 0.$$  We therefore have
$$ \bor q  = \bor q_1 + \bor(q_2 + \cdots +q_r) 
= x_{i_1}x_{i_2}  + \mbox{ terms which are } \geq x_{i_2}x_{j_2}. $$
  So the next smallest mixed monomial in $p$ is the smallest mixed monomial of $q_2 + \cdots + q_r$, that is $x_{i_2}x_{ j_2}$. So $p_2 = q_2$.

{\bf Case $t > 1$:}  Our induction hypothesis is that, for $s < t$, $p_s = q_s$. The normality of $p$ forces $\coef(x_{i_s}x_y, p) = 0$ for all $y \not= j_s$ and $\coef(x_{j_s}x_y, p) = 0$ for all $y \not= i_s$. In addition, for $1 \leq s < t$ we assume that $b_{j_s y} = 0 $ holds for all $y \not= j_1, \ldots, j_s$. Then
\begin{align}\label{eq2} \coef(x_{i_t}x_y,p) & = b_{i_1 i_t}b_{j_1 y} + \cdots + b_{i_t i_t}b_{j_t y} = \begin{cases} b_{i_t i_t}b_{j_t y}  = 0 & \text{ for } y \not= j_1, \ldots, j_{t-1} \\
  b_{i_t i_t}b_{j_t j_t} = 1 & \text{ for } y = j_t \end{cases} \\ & \Rightarrow \ \ \ \ b_{j_t y} = 0 \mbox{ for all } y \not= j_1, \ldots, j_t.\end{align}

Now consider the mixed monomial $x_yx_z$ with $x_yx_z$ being a summand in $b \sum_{s < t}x_{i_{s}}x_{j_{s}}$ but not in $b x_{i_t}x_{j_t}$. In other words $(y,z) \prec (i_{s},j_{s})$ for $s<t$, but $(y,z) \not\prec (i_t,j_t)$. Additionally $y \not= z$, $y,z \not\in \{j_1,, \ldots, j_t\}$. Then 
$$\coef(x_yx_z,p) = \sum_{l\leq t}b_{i_l y}\underbrace{b_{j_l z}}_{=0} +  \sum_{l\leq t}b_{i_l z}\underbrace{b_{j_l y}}_{=0} = 0.$$
So the next smallest mixed monomial in $p$ is $x_{i_t}x_{j_t}$ and $p_t = q_t$. By induction it follows that $p$ and $q$ have the same mixed monomials.

For $q= q_1+ \cdots + q_r$ with $q_t = \epsilon_tx_{i_t}^2 + x_{i_t}x_{j_t}$ for each $t$, we have $p = p_1 + \cdots + p_r$ where $p_t = \alpha_tx_{i_t}^2 + x_{i_t}x_{j_t}$ with $\alpha_t \in \{0,1\}$ for all $1 \leq t \leq r$.
We now wish to show that $\epsilon_t = \alpha_t$ for all $t$. The proof is by induction on $t$.

{\bf Case $t=1$:} Since $$\bor q = \epsilon_1\bor x_{i_1}^2 + \bor x_{i_1}x_{j_1} + \mbox{higher terms} = \epsilon_1\underbrace{b_{i_1}^2}_{\not= 0}x_{i+1}^2 + \mbox{ higher terms }$$ then $x_{i_1}^2$ is a summand of $p$ if and only if $\epsilon_1 = 1$. So $\epsilon_1 = \alpha_1$.

{\bf Case $t > 1$:} We assume $\alpha_s = \epsilon_s$ for all $s < t$. Since $p = \bor q$, we have $$\alpha_t = \sum_{s\leq t} \epsilon_s b_{i_s i_t}^2 + \sum_{s < t}b_{i_s i_t}\underbrace{b_{j_s i_t}}_{=0} =  \sum_{s\leq t} \epsilon_s b_{i_s i_t}^2.$$ If there is a $j_k$ ($k < t$) with $j_k> j_t$ then the pairs $(i_k, j_k)$, $(i_t,j_t)$ satisfy $i_k < i_t < j_t < j_k$. Since $q$ is normal, by (C2) we have $\epsilon_k\epsilon_t = 0$. Likewise this condition holds for $p$ and we have
$$0 = \alpha_k\alpha_t = \underbrace{\epsilon_k}_{\mbox{by induction}}\left(\sum_{s\leq t} \epsilon_sb_{i_s i_t}^2\right) = \epsilon_k\sum_{s< t} \epsilon_sb_{i_s i_t}^2 + \underbrace{\epsilon_k\epsilon_t}_{=0}b_{i_t i_t}^2 = \epsilon_k\sum_{s< t} \epsilon_sb_{i_s i_t}^2.$$
If $\epsilon_k = 1$ then $\sum_{s< t} \epsilon_sb_{i_s i_t}^2 = 0$. Then $\alpha_t = \epsilon_tb_{i_t i_t}^2 = \epsilon_t$ (by normality) and we are done. Notice than, in order to conclude, it sufficed to find only one $k < t$ such that $j_k > j_t$ and $\epsilon_k = 1$.

If we are not able to use this argument, it means that for each $k<t$, either $j_k < j_t$ or $\epsilon_k = 0$. In other words, the only possibly non-zero summands in $\sum_{s \leq t} \epsilon_sb_{i_s i_t}^2$ are those for index $k<t$ with $j_k < j_t$. We show then that $b_{i_k i_t} = 0$. The proof is by induction on $t-k$.

{\bf Case $t-k=1$:}  We have \begin{align} 0 = \coef(x_{i_t}x_{j_{t-1}},p)& = \sum_{k \leq t}b_{i_ki_t}\underbrace{b_{j_k j_{t-1}}}_{=0\mbox{ for }k<t-1} \\ &= b_{i_{t-1}i_t}b_{j_{t-1}j_{t-1}} + b_{i_t i_t} b_{j_t j_{k-1}}. \end{align}
If $j_{t-1} < j_t$ then $b_{j_t j_{k-1}} = 0$ and therefore $b_{i_{t-1}i_t} = 0$.

{\bf Case $t-k >1$:}  Now assume that the claim holds for all $t-k < m$. We have
\begin{align} 0 = \coef(x_tx_{j_{t-m}}, p) & = \sum_{k \leq t} b_{i_k i_t} \underbrace{b_{j_k j_{t-m}}}_{ = 0 \mbox{ for } k<t-m} \\ & = b_{i_{t-m} i_t}b_{j_{t-m} j_{t-m}} + \sum_{t-m < k < t} b_{i_k i_t}b_{j_k j_{t-m}} + b_{i_t i_t}b_{j_t j_{t-m}} \end{align}

Let $j_{t-m} < j_t$. For each $t-m < k < t$ we have either $j_{t-m} < j_k$ or $j_k < j_t$. In the former case $b_{j_k j_{t-m}} = 0$ holds. In the latter case $b_{i_k i_t} = 0$ holds by induction. So each summand $b_{i_k i_t} b_{j_k j_{t-m}} = 0$. Therefore
  $$  0 = \coef(x_tx_{j_{t-m}},p) = b_{i_{t-m}i_t}b_{j_{t-m}j_{t-m}} + b_{i_t i_t}\underbrace{b_{j_t}b_{j_{t-m}}}_{=0}.$$
  So $b_{i_{t-m} i_t} = 0$. With that our claim holds. It follows that $\alpha_t = \epsilon_t$ for all $t$. We conclude that $p=q$. If we had indeed considered $\tilde{q}$ instead of $q$ then $p = \tilde{p_|}_{x_n=0} = \tilde{q_|}_{x_n=0} = q$. 

  {\bf Proof of 2.} From the above proof for uniqueness, since $p=q$ then $b \in B_q$. Claim (1) follows from (\ref{eq1}) and (\ref{eq2}). Claim (2) is exactly what is proved by induction on $t-k$.
\end{proof}

Now that we have characterized particular representatives of the Borel orbits of quadratic forms, we would like to determine their rank and those which are non-degenerate.
\begin{lem} A normal form $q$ is non-degenerate if and only if  $\ind(q) =\{1, \ldots, n\}$. \end{lem}
\begin{proof} $(\Rightarrow )$ Let $\be_t = (0, \ldots, 0,1,0,\ldots,)$ with $t \not\in \ind(q)$. Then $q(\be_t) = 0$ although $\be_t \not=0$. So $q$ is degenerate.

  $(\Leftarrow)$  First consider the case when $n$ is even. Notice that $q$ can't contain a normal component which is a pure square $q_s = x_{i_s}^2$. If it did, this would be unique (due to normality), but then $\ind(q)$ would contain an odd number of elements, contradicting $\ind(q) =\{1, \ldots, n\}$. So $\ind(q)$ is a union of index pairs. Let ${ l}$ be the symmetric bilinear form associated to $q$. Let $M = (m_{i,j})$ be the matrix associated to ${ l}$. Then $ m_{i,j} = 1$ if $(i,j)$ or $(j,i)$ is an index pair for $q$. Otherwise $m_{i,j}=0$. Since  $\ind(q) =\{1, \ldots, n\}$ then every row and every column of $M$ has exactly one non-zero entry, which is $1$. So $M$ is invertible and $l$ is non-degenerate and, therefore,  $q$ is non-degenerate. If $n$ is odd, then there is a unique $t$ with $q_t = x_{i_t}^2$. In this case the $t^{\mbox{th}}$ row and column of $M$ are $0$-vectors. The $(n-1)\times (n-1)$ submatrix consisting of $M$ without the  $t^{\mbox{th}}$ row and column is non-degenerate. So $\KK  \be_t = \{ v \ \mid \ l(v,w) = 0 \mbox{ for all } w \in \KK^n\ \}$. However $q(\be_t)  = 1 \not= 0$ so $q$ is non-degenerate (see \cite{conrad2020bilinear}, Thm. 7.3).
  \end{proof}

  We now consider those quadratic forms $q$ for which $Bq$ is maximal in a certain way. Let $B_q$ be the stabilizer of $q$ under the action of $B$. Let $\pi: B \rightarrow T $ be the projection on the maximal torus $T$. So for $\bor \in B$ we have $\pi(b)_{i,j}= b_{i,i}$ for $i=j$ else $\pi(\bor)_{i,j}= 0$.
  \begin{defn} The {\bf $B$-rank } of a $B$-orbit $Bq$ is $n - \dim(\pi(B_q))$.

    \end{defn}

                                      \begin{lem}\label{lem-brank} For $q = q_1 + \cdots + q_r$ normal let  \begin{align} a_1&:= |\{q_t \ \mid \ q_t = x_{i_t}x_{j_t} \ \}|, \\  
                                        a_2 &:= |\{q_t \ \mid \ q_t = x_{i_t}^2+x_{i_t}x_{j_t} \ \}|, \\  
                                        a_3 & := |\{q_t \ \mid \ q_t = x_{i_t }^2 \ \}|. \end{align}
                                      For $B \subset GL(n,\KK)$ we have $ \rank(Bq) = a_1 + 2a_2 + a_3$.
                                   
                                    \end{lem}
                                    \begin{proof} Let $\bor \in B_q$. If $ q_t = x_{i_t }^2$ then, by claim 2. of Theorem~\ref{lem-borbit}, we have $$1 = \coef (x_{i_t}^2,\bor q) = \sum_{s \leq t}\underbrace{\epsilon_s b_{i_s i_t}}_{= 0 \mbox{ for } s<t}^2 + \sum_{s\leq t}\delta_s b_{i_s i_t}\underbrace{b_{j_s i_t}}_{=0} = b_{i_s i_s}^2$$
                                      so $b_{i_s i_s} = 1$. Similarly one obtains $1 = \coef(x_{i_t}x_{j_t},\bor q) = b_{i_t i_t}b_{j_t j_t}$. It follows that 
   $$ b_{i_t i_t} = b_{j_t j_t}^{-1}   = \begin{cases} 1 & \text{ if } \epsilon_t = 1 \\ \in \KK^* & \text{ if } \epsilon_t = 0 \end{cases}. $$
                                The dimension of $\pi(B_q)$ is the number of degrees of freedom of the diagonal coefficients, that is $ n- (a_1 + 2a_s + a_3)$. The claim follows.
                                    \end{proof}
                                   
                                    \begin{lem}\label{lem-orbcatalan}
                                 
                                      \begin{enumerate} \item Let $q$ be normal and of maximal $B$-rank $n$. Then $q = q_1 + \cdots + q_r + \epsilon q_{r+1}$ where
                                      \begin{enumerate}
                                      \item $r = \lfloor \frac{n}{2} \rfloor$
                                      \item $\epsilon = 1 $ if and only if $n$ is odd
                                      \item $q_t = x_{i_t}^2 + x_{i_t}x_{j_t}$ for every $1 \leq t \leq r$ and $q_{r+1} = x_{i_{r+1}}^2$
                                        \end{enumerate}
                                     \item The number of Borel orbits $Bq$ of maximal $B$-rank $n$ equals $C_{\lfloor \frac{n+1}{2}\rfloor}$
where $C_m$ denotes the $m$th Catalan number. \end{enumerate}
                                    \end{lem}
                                    \begin{proof} Wir first consider the case when $n = 2r$ is even. Then by Lemma~\ref{lem-brank}, the  $B$-rank of $Bq$ is maximized when when $a_1 = 0$, $a_2 = r$, and $a_3=0$. This forces $q = q_1 + \cdots + q_r$ where $q_t = x_{i_t}^2 + x_{i_t}x_{j_t}$ for each $t$. Since $q$ is normal, it has no nested index pairs $(i_k, j_k)$, $(i_l,j_l)$ with $i_k < i_l < j_l < j_k$. The quadratic form $q$ is normal and of maximal  $B$-rank if and only if its diagram represents a non-nesting matching of the numbers $\{1, \ldots, n\}$. The number of such matchings is $C_r$, see
                                      \cite[p.29]{stanley2015catalan}. 

                                      In the case that $n = 2r+1$ is odd, $Bq$ has maximal  $B$-rank $n = 2r+1$  when $q$ is of the form $q = q_1 + \cdots + q_r + q_{r+1}$ where $q_t = x_{i_t}^2 + x_{i_t}x_{j_t}$ for each $t \leq r$ and $q_{r+1} = x_{i_{r+1}}^2$. The number of such $q$ is equal to the number of extended quadratic forms $\tilde{q}$ which, using the same argument as above, is $C_{r+1}$. 
                                    \end{proof}

                              \begin{defn}        Let $q =q_1 + \cdots + q_r + \epsilon x_{i_{r+1}}^2$ be a \emph{non-degenerate} quadratic form of maximal $B$-rank $n$. Let us first consider the case for $n = 2r$ even, that is $\epsilon = 0$. We say that $q':= q_s + q_{s+1} \cdots + q_{s+t}$ is a {\bf connected component} of $q$ if $t \geq 0$ is minimal with the property that for every $i \in \ind(q')$ follows $$ \max(\ind(q_1 + \cdots +q_{s-1})) < i < \min(\ind(q_{s+t+1}+ \cdots + q_r)).$$
                                If $q$ has only one connected component, then it is {\bf connected}.  Otherwise it is {\bf disconnected}. In the case that $n = 2r+1$ is odd, we say that $q$ is connected if and only if $\tilde{q}$ is connected (as defined for $n$ even). We denote the number of connected components in $q$ by $\ccomp(q)$.\end{defn}
                              \begin{rem}  The connectedness of a quadratic form $q$ is easily seen by its diagram. If it has no isolated point, then the diagram is connected if every vertical line crossing the diagram touches an edge. The quadratic form $q$ with diagram \begin{tikzpicture}[scale=0.8]
  \tikzset{vertex/.style={draw,circle,inner sep=1pt, fill=black}}
   \tikzset{overtex/.style={draw,circle,inner sep=1pt}}
    
    \node[vertex] (1) at (0,0) { };
    \node[vertex] (2) at (1,0) { };
    \node[overtex] (3) at (2,0) { };

    \draw (1) to [out=30,in=150] (3);
   
  \end{tikzpicture} is connected because the diagram for $\tilde{q}$ is connected:

  \begin{center} \begin{tikzpicture}[scale=0.8]
  \tikzset{vertex/.style={draw,circle,inner sep=1pt, fill=black}}
   \tikzset{overtex/.style={draw,circle,inner sep=1pt}}
    
    \node[vertex] (1) at (0,0) { };
    \node[vertex] (2) at (1,0) { };
    \node[overtex] (3) at (2,0) { };
    \node[overtex] (4) at (3,0) { };

    \draw (1) to [out=30,in=150] (3);
    \draw (2) to [out=30,in=150] (4);
\end{tikzpicture} \end{center}

 Connected components of the diagram are (maximally) connected subdiagrams. If it has an isolated point, we consider the diagram for the extended $\tilde{q}$. So the diagram \ref{diag1} has two connected components because the diagram below for its extension $\tilde{q}$ has two components:
  
  
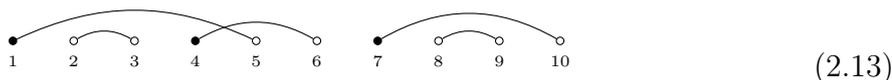
\begin{equation}
  \begin{tikzpicture}[scale=0.8]
  \tikzset{vertex/.style={draw,circle,inner sep=1pt, fill=black}}
   \tikzset{overtex/.style={draw,circle,inner sep=1pt}}
    
    \node[vertex,label=below:{\tiny $1$}] (1) at (0,0) { };
    \node[overtex,label=below:{\tiny $2$}] (2) at (1,0) { };
    \node[overtex,label=below:{\tiny $3$}] (3) at (2,0) { };
    \node[vertex,label=below:{\tiny $4$}] (4) at (3,0) { };
     \node[overtex,label=below:{\tiny $5$}] (5) at (4,0) { };
    \node[overtex,label=below:{\tiny $6$}] (6) at (5,0) { };
    \node[vertex,label=below:{\tiny $7$}] (7) at (6,0) { };
    \node[overtex,label=below:{\tiny $8$}] (8) at (7,0) { };
    \node[overtex,label=below:{\tiny $9$}] (9) at (8,0) { };
     \node[overtex,label=below:{\tiny $10$}] (10) at (9,0) { };

    \draw (1) to [out=25,in=155] (5);
    \draw (2) to [out=30,in=150] (3);
     \draw (4) to [out=30,in=150] (6);
     \draw (8) to [out=30,in=150] (9);
       \draw (7) to [out=30,in=150] (10);
  \end{tikzpicture}
\end{equation}
                                       \end{rem}
                                     \begin{defn}
                                        
                                    We denote by $ \bc(n,f)$ the number of non-degenerate, maximal $B$-rank $n$ quadratic forms with $f$ connected components. \end{defn}  \begin{rem}\label{rem-orbcatalan} For $n=2r$ we have $\bc(2r,f) = 0$ for $f > r$ or $f<0$ and $\bc(2r,r) = 1$ (the case where the corresponding diagram is \ \ 
 \begin{tikzpicture}[scale=0.5]
\filldraw
(0,0) circle (2pt)
(2,0) circle (2pt)
(6,0) circle (2pt);
\draw
(1,0) circle (2pt)
(3,0) circle (2pt)
(7,0) circle (2pt);
\node at (4,0) {\ldots};
\node at (5,0) {\ldots};
\uarc01 \uarc23 \uarc67
\end{tikzpicture} ). From Lemma~\ref{lem-orbcatalan} we know that $\sum_{f=1}^r\bc(2r,f) = C_{r}$, the Catalan number.

To count $\bc(2r,f)$ we consider the diagrams associated to quadratic forms. To such a diagram we can associate a (horizontal) Dyck path or ``mountain range''\cite{BERNHART1973} in which at each $i_t$ there is an upward stroke and at each $j_t$ there is a downward stroke.
                                         That $i_1=1$ and $j_r = n$ means that the range moves up from the ``ground'' or $0$-level and that it ends at the ground. The condition $t < j_t $ implies that the mountain never goes below the ground.
                                         If the diagram has $f$ connected components, then the first component is over the first, say, $2l$ dots. Such a component corresponds to a mountain range

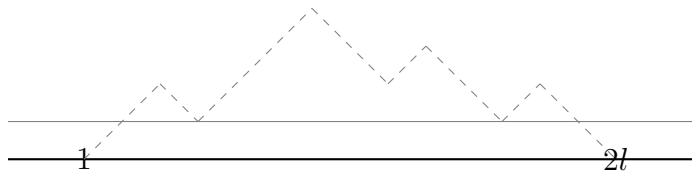
\begin{center}
  \begin{tikzpicture}
     \node (1) at (1,0) {$1$};
    \node (2) at (8,0) {$2l$};
  \draw[black,thick] (0,0) -- (9,0);
  \draw[gray, thin] (0,0.5) -- (9,0.5);
  \draw[gray, dashed] (1,0) -- (2,1);
  \draw[gray, dashed] (2,1) -- (2.5,0.5);
   \draw[gray, dashed] (2.5,0.5) -- (4,2);
   \draw[gray,dashed] (4,2) -- (5,1);
   \draw[gray,dashed] (5,1) -- (5.5,1.5);
    \draw[gray,dashed] (5.5,1.5) -- (6.5, 0.5);
    \draw[gray,dashed] (6.5,0.5) -- (7,1);
    \draw[gray,dashed]  (7,1) -- (8,0);
   
  \end{tikzpicture}
\end{center}

\noindent
which, except for at vertices $1$ and $2l$, never touches the ``ground'' or the $0$-line. The above mountain range (with $2l =14$) corresponds to the sequence and diagram

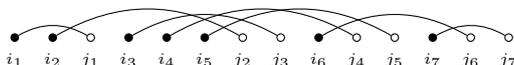
\begin{center}
   \begin{tikzpicture}[scale=0.5]
  \tikzset{vertex/.style={draw,circle,inner sep=1pt, fill=black}}
   \tikzset{overtex/.style={draw,circle,inner sep=1pt}}
    
    \node[vertex,label=below:{\tiny $i_1$}] (1) at (0,0) { };
    \node[vertex,label=below:{\tiny $i_2$}] (2) at (1,0) { };
    \node[overtex,label=below:{\tiny $j_1$}] (3) at (2,0) { };
    \node[vertex,label=below:{\tiny $i_3$}] (4) at (3,0) { };
     \node[vertex,label=below:{\tiny $i_4$}] (5) at (4,0) { };
    \node[vertex,label=below:{\tiny $i_5$}] (6) at (5,0) { };
    \node[overtex,label=below:{\tiny $j_2$}] (7) at (6,0) { };
    \node[overtex,label=below:{\tiny $j_3$}] (8) at (7,0) { };
     \node[vertex,label=below:{\tiny $i_6$}] (9) at (8,0) { };
    \node[overtex,label=below:{\tiny $j_4$}] (10) at (9,0) { };
    \node[overtex,label=below:{\tiny $j_5$}] (11) at (10,0) { };
    \node[vertex,label=below:{\tiny $i_7$}] (12) at (11,0) { };
     \node[overtex,label=below:{\tiny $j_6$}] (13) at (12,0) { };
    \node[overtex,label=below:{\tiny $j_7$}] (14) at (13,0) { };
   
    \draw (1) to [out=30,in=150] (3);
    \draw (2) to [out=30,in=150] (7);
      \draw (4) to [out=30,in=150] (8);
      \draw (5) to [out=30,in=150] (10);
        \draw (6) to [out=30,in=150] (11);
        \draw (12) to [out=30,in=150] (14);
          \draw (9) to [out=30,in=150] (13);
  
  \end{tikzpicture}
  \end{center} 
  There are $C_{l-1}$ different ways of creating a mountain range over the $2(l-2)$ vertices $2, \ldots, 2l-1$. There are $\bc(2r-2l,f-1)$ ways of creating $f-1$ mountain ranges over the $2r-2l$ remaining vertices. Since $l$ can range from $1$ to $r-f+1$ we have the recursive relation
$$\bc(2r,f) = \sum_{l=1}^{r-f+1}C_{l-1}\cdot \bc(2r-2l,f-1).$$
In particular, $\bc(2r, 1) = C_{r-1}$.

In the case that $n = 2r+1$ is odd, then every non-degenerate quadratic form $q = \sum_{l=1}^r(x_{i_l}^2 + x_{i_l}x_{j_l}) + x_{i_{r+1}}^2$ of maximal $B$-rank $2r$ is normal if and only if its extension $\tilde{q} =\sum_{l=1}^r(x_{i_l}^2 + x_{i_l}x_{j_l}) + x_{i_{r+1}}^2 + x_{i_{r+1}}x_{n+1}$ is normal. From the discussion in the previous paragraph, we have then that
$$\bc(2r+1,f) = \bc(2r+2,f) =  \sum_{l=1}^{r-f+2}C_{l-1}\cdot \bc(2r-2l+2,f-1).$$
\end{rem}
The {\it Catalan triangle numbers}, a generalization of the Catalan numbers, were first introduced by Shapiro in \cite{shapiro1976catalan}. Just like binomial coefficients, they can be defined recursively \cite{lee2016catalan}.
\begin{defn} The {\bf $(n,k)$-Catalan triangle number $C(n,k)$} is defined by $C(n,k) = 0 $ for $k>n$ or $n<0$ and for $n\geq 0$ and $0 \leq k \leq n$
  $$C(n,k) = \begin{cases} 1 & \text{ for } n=k=0; \\ C(n,k-1)+C(n-1,k) & \text{ for } 0 < k < n; \\ C(n-1,0) & \text{ for } k=0; \\ C(n,n-1) & \text{ for } k=n, \end{cases}.$$
  
In particular, for all $n\geq 0$, we have $C(n,n) = C_n$, the $n^{\mbox{th}}$ Catalan number.
  \end{defn} 
 \begin{lem} The number $\bc(2r,f)$  of non-degenerate quadratic forms of maximal $B$-rank $2r$ with $f$ connected components equals the Catalan triangle number $C(r-1,r-f)$.

                                      \end{lem}
                                      \begin{proof}
                                        We show that $C(n,k) = \bc(2(n+1),n-k+1)$ satisfies the necessary defining conditions of a Catalan triangle number \cite{shapiro1976catalan}.

                                        \begin{enumerate} \item {\bf $C(n,0) = 1$ for all $n\geq 0$:} From the previous remark we know that $\bc(2n+2,n+1) = 1 $ for all $n \geq 0$ so $\bc(2n+2,n+1) = C(n,0)$. 
                                        \item {\bf $C(n,1) = n $ for $n \geq 1$:} The proof is by induction on $n$. We clearly have $\bc(4,2) = 1 = C(1,1)$. Since
                                          \begin{align} \bc(2n+2,n) & = \sum_{l=1}^2C_{l-1}\cdot \bc(2(n+1-l),n-1) \\ &= C_0\cdot \bc(2n,n-1) + C_1\cdot \bc(2n-2,n-1) \\ & = 1 \cdot \underbrace{(n-1)}_{\mbox{ (by induction)}} + 1 \cdot 1 = n \end{align}
                                          we conclude $\bc(2n+2,n) = C(n,1)$ for all $n$.
                                        \item {\bf $C_{n+1} = C(n+1,n+1) = C(n+1,n)$ for $n\geq 0$:} We have $\bc(2(n+2), 1) = C_{n+1}$ from Remark~\ref{rem-orbcatalan} for all $n \geq 0$. In addition, \begin{eqnarray} \bc(2(n+2), 2) = \sum_{l=1}^{n+1}C_{l-1}\cdot \bc(2(n+2-l),1) = \sum_{l=1}^{n+1}C_{l-1}\cdot C_{n+1-l} = C_{n+1} \end{eqnarray}
                                          It follows that $\bc(2n+2,1) = C(n+1,n+1) = C(n+1,n) = \bc(2n+2,2)$ for all $n$.
                                        \item {\bf $C(n+1,k) = C(n+1,k-1) +C(n,k)$ for $1 < k < n+1$: } The proof is by induction on $n +k $. For $n+k=3$ we have
                                          $ \bc(8,2) = C_3 = 5 = C(3,2)$. In addition $$ \bc(8,3) = \sum_{l=1}^2C_{l-1}\bc(2(4-l),2) = C_0\bc(6,2) + C_1\bc(4,2) = C_0C_2 + C_1^2 = 3 = C(3,1),$$ and $\bc(6,1) = C_2 =2=  C(2,2)$. So the equation is satisfied in this case. In general,
                                          \begin{align} \bc(2(n+2),&n-k+2)  - \bc(2n+2,n-k+1) \\ &  = \sum_{l=1}^{k+1}C_{l-1}\left( \bc(2(n+2-l),n-k+1) - \bc(2(n+1-l),n-k)\right) \\
                                            \mbox{ (by induction) } \hskip .5cm  & =  \sum_{l=1}^{k+1}C_{l-1}\cdot \bc(2(n+2-l),n-k+2)  = \bc(2(n+2),n-k+3)\end{align}
                                                                                          and since (by induction) $\bc(2(n+2),n-k+3) = C(n+1,k-1)$ and $\bc(2n+2,n-k+1) = C(n,k)$ then $ \bc(2(n+2),n-k+2) = C(n+1,k)$. 

                                                                                        \end{enumerate}
                                      \end{proof}
                                      \begin{rem} With that we obtain a representation theoretic proof of a generalization of the recursive Catalan number identity $$C_{n+1} = \sum_{l=1}^nC_lC_{n-l}, \ \ \ C_0=1$$ to the Catalan triangle numbers:

                                        \end{rem}
                                      \begin{cor}(\cite{hilton1991catalan},\cite{vera2014generalization}) The Catalan triangle numbers satisfy the recursive relation
                                        $$C(n,k) = \sum_{l=0}^kC(l,l)C(n-l-1,k-l)$$ for $k<n$.

                                        \end{cor}

                                        \section{Parabolic orbits}
                                        
                                        Let $P:=P_i$ be the minimal parabolic subgroup $\langle B, s_{i} \rangle$ of $G$ where $s_i$ is the simple reflection $(i\ i+1)$. In this section we will investigate how the $P$-orbit $Pq$ of a normal quadratic form $q$ decomposes as a union of $B$-orbits. This information can then be encoded in the so-called Brion graph \cite{brion2001orbit}.

                                        \begin{rem} Let $q$ be a normal quadratic form. 
                                        The action of $s_i$ on $q$ is determined by its action on those normal components of $q$ which contain $i$ and $i+1$; it switches $x_i$ and $x_{i+1}$ and fixes all other $x_j$. The resulting quadratic form $s_iq$ can be also normal (but need not be) and it could have a different rank from $q$. 
                                        \end{rem}
\begin{lem}\label{lem-porbit}
 Let $q = q'+ p$ be a normal quadratic form where $p$ is the sum of all normal components which do not contain $i, i+1$. The $P$-orbit $Pq$ decomposes as a union of $B$-orbits in the following way:
  \begin{enumerate}[leftmargin=*]
  \item if $q'= 0$ (i.e. $i,i+1 \not\in \ind(q)$) then $P q = B q$,
  \item else if $q = \epsilon x_{i}^2 + x_{i}x_{i+1} + p$, $\epsilon \in \{0,1\}$ then  $P q = B (x_{i}^2 + x_{i}x_{i+1} + p) \cup B( x_{i}x_{i+1} +p)$
    \item else if  $q \in M_1:=\{ \epsilon_j x_i^2 + x_ix_j +\epsilon_{k}x_{i+1}^2+x_{i+1}x_k+p \ | \ \epsilon_{\max(j,k)}=1\} $ for some $j,k \not=i,i+1$ then  $P q= \cup_{g \in M_1}B g $, \\
    OR  $q \in M_2:=\{ \epsilon_jx_j^2 + x_jx_i + \epsilon_i x_{k}^2+x_{k}x_{i+1} +p \ | \ \epsilon_{\min(j,k)}=1 \}$   for some $j,k \not=i,i+1$  then  $P q= \cup_{g \in M_2}B g$, 
  \item otherwise $s_iq = q'$ is normal and  $P q = B q \cup Bq'  $.

    \end{enumerate} 
  \end{lem}
  \begin{proof} \begin{enumerate} \item Follows from $s_iq = q$.

    \item So the component in $q$ containing $i,i+1$ has as diagram
      \raisebox{-.25\height}{\begin{tikzpicture}[scale=0.8]
  \tikzset{vertex/.style={draw,circle,inner sep=1pt, fill=black}}
   \tikzset{overtex/.style={draw,circle,inner sep=1pt}}
    
    \node[vertex,label=below:{\tiny $i$}] (1) at (0,0) { };
    \node[overtex,label=below:{\tiny $i+1$}] (2) at (1,0) { };

    \draw (1) to [out=30,in=150] (2);
  
  \end{tikzpicture}} or
 \raisebox{-.25\height}{\begin{tikzpicture}[scale=0.8]
  \tikzset{vertex/.style={draw,circle,inner sep=1pt, fill=black}}
   \tikzset{overtex/.style={draw,circle,inner sep=1pt}}
    
    \node[overtex,label=below:{\tiny $i$}] (1) at (0,0) { };
    \node[overtex,label=below:{\tiny $i+1$}] (2) at (1,0) { };

    \draw (1) to [out=30,in=150] (2);
   
\end{tikzpicture}}. We have  $s_{i}(x_{i}^2+ x_{i}x_{i+1}) = x_{i+1}^2 + x_{i}x_{i+1}$, which is not normal but whose normal form is $x_ix_{i+1}$. The claim follows.

\item  In the first case the components in $q$ containing $i,i+1$ have as diagram

  \begin{center}  \raisebox{-.25\height}
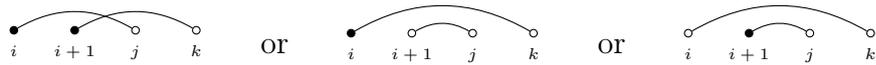
{\begin{tikzpicture}[scale=0.8]
  \tikzset{vertex/.style={draw,circle,inner sep=1pt, fill=black}}
   \tikzset{overtex/.style={draw,circle,inner sep=1pt}}
    
    \node[vertex,label=below:{\tiny $i$}] (1) at (0,0) { };
    \node[vertex,label=below:{\tiny $i+1$}] (2) at (1,0) { };
    \node[overtex,label=below:{\tiny $j$}] (3) at (2,0) { };
    \node[overtex,label=below:{\tiny $k$}] (4) at (3,0) { };

    \draw (1) to [out=30,in=150] (3);
    \draw (2) to [out=30,in=150] (4);
\end{tikzpicture}} \hskip .5cm or \hskip .5cm \raisebox{-.25\height}{\begin{tikzpicture}[scale=0.8]
  \tikzset{vertex/.style={draw,circle,inner sep=1pt, fill=black}}
   \tikzset{overtex/.style={draw,circle,inner sep=1pt}}
    
     \node[vertex,label=below:{\tiny $i$}] (1) at (0,0) { };
    \node[overtex,label=below:{\tiny $i+1$}] (2) at (1,0) { };
    \node[overtex,label=below:{\tiny $j$}] (3) at (2,0) { };
    \node[overtex,label=below:{\tiny $k$}] (4) at (3,0) { };

    \draw (1) to [out=30,in=150] (4);
    \draw (2) to [out=30,in=150] (3);
\end{tikzpicture}} \hskip .5cm or \hskip .5cm \raisebox{-.25\height}{\begin{tikzpicture}[scale=0.8]
  \tikzset{vertex/.style={draw,circle,inner sep=1pt, fill=black}}
   \tikzset{overtex/.style={draw,circle,inner sep=1pt}}
    
    \node[overtex,label=below:{\tiny $i$}] (1) at (0,0) { };
    \node[vertex,label=below:{\tiny $i+1$}] (2) at (1,0) { };
    \node[overtex,label=below:{\tiny $j$}] (3) at (2,0) { };
    \node[overtex,label=below:{\tiny $k$}] (4) at (3,0) { };

    \draw (1) to [out=30,in=150] (4);
    \draw (2) to [out=30,in=150] (3);
\end{tikzpicture}} \end{center} 
We have $s_i(x_i^2 + x_ix_j +x_{i+1}^2+x_{i+1}x_k) = x_{i+1}^2 + x_{i+1}x_j + x_i^2 + x_ix_k$, which is not normal because $(i,k)$ and $(i+1,j)$ are nested. The normal form of $ x_{i+1}^2 + x_{i+1}x_j + x_i^2 + x_ix_k$ is $ x_i^2+x_ix_k+x_{i+1}x_j$. In addition, the idempotent $s_i$ switches $x_i^2+x_ix_k+x_{i+1}x_j$ and $x_{i+1}^2+x_{i+1}x_k+x_ix_j$, both normal but not of maximal rank.

In the second case  the components in $q$ containing $i,i+1$ have as diagram

  \begin{center} \raisebox{-.25\height}
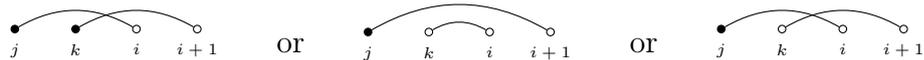
{\begin{tikzpicture}[scale=0.8]
  \tikzset{vertex/.style={draw,circle,inner sep=1pt, fill=black}}
   \tikzset{overtex/.style={draw,circle,inner sep=1pt}}
    
    \node[vertex,label=below:{\tiny $j$}] (1) at (0,0) { };
    \node[vertex,label=below:{\tiny $k$}] (2) at (1,0) { };
    \node[overtex,label=below:{\tiny $i$}] (3) at (2,0) { };
    \node[overtex,label=below:{\tiny $i+1$}] (4) at (3,0) { };

    \draw (1) to [out=30,in=150] (3);
    \draw (2) to [out=30,in=150] (4);
\end{tikzpicture}}  \hskip .5cm or \hskip .5cm  \raisebox{-.25\height}{\begin{tikzpicture}[scale=0.8]
  \tikzset{vertex/.style={draw,circle,inner sep=1pt, fill=black}}
   \tikzset{overtex/.style={draw,circle,inner sep=1pt}}
    
     \node[vertex,label=below:{\tiny $j$}] (1) at (0,0) { };
    \node[overtex,label=below:{\tiny $k$}] (2) at (1,0) { };
    \node[overtex,label=below:{\tiny $i$}] (3) at (2,0) { };
    \node[overtex,label=below:{\tiny $i+1$}] (4) at (3,0) { };

    \draw (1) to [out=30,in=150] (4);
    \draw (2) to [out=30,in=150] (3);
\end{tikzpicture}}  \hskip .5cm or \hskip .5cm  \raisebox{-.25\height}{\begin{tikzpicture}[scale=0.8]
  \tikzset{vertex/.style={draw,circle,inner sep=1pt, fill=black}}
   \tikzset{overtex/.style={draw,circle,inner sep=1pt}}
    
    \node[vertex,label=below:{\tiny $j$}] (1) at (0,0) { };
    \node[overtex,label=below:{\tiny $k$}] (2) at (1,0) { };
    \node[overtex,label=below:{\tiny $i$}] (3) at (2,0) { };
    \node[overtex,label=below:{\tiny $i+1$}] (4) at (3,0) { };

    \draw (1) to [out=30,in=150] (3);
    \draw (2) to [out=30,in=150] (4);
\end{tikzpicture}} \end{center} The proof is analogous to the one for the first case.
  
         \item Follows from the idempotent action of $s_i$ which switches two normal forms $q$ and $q'$.
    \end{enumerate}
 \end{proof}
 \begin{rem}  Let $P:= \langle s_i,B \rangle$, the simple parabolic subgroup, let $q$ be normal and $P_q$ the stabilizer of $q$ in $P$. We denote by $\Phi(P_q)$, the group of $2\times 2$ matrices:
   $$\Phi(P_q) := \left\{ \begin{pmatrix} y_{i,i} & y_{i,i+1} \\ y_{i+1,i} & y_{i+1,i+1} \end{pmatrix} \ \mid \ y \in P_q \ \right\}\backslash \KK^*\Bone_2 \subseteq PGL(2,\KK).$$  Clearly $\Phi(P_q)$ is determined by the normal component(s) of $q$ containing $x_i$ and $x_{i+1}$. This group will be conjugate to one of the following groups \cite{knop1995set}:
   \begin{enumerate}
   \item $G_0 := PGL(2,\KK)$, 
   \item $T_0:=T(2,\KK) \subset G_0$, the group of diagonal matrices.
   \item $N_0 := T_0 \cup  \begin{pmatrix} 0 & 1 \\ 1 & 0 \end{pmatrix}T_0 \subset G_0$,
   \item $U_0 := U(2,\KK) \subset G_0$, the unitary group of upper-triangular matrices with diagonal entries equal to $1$.
     \end{enumerate}

 \end{rem}
 \begin{cor}\label{lem-pstab} Let $q$ be normal. Then $\Phi(P_q)$  is as follows.
   \begin{center}
     \begin{tabular}{ccc}
     Case&   Conditions on $q$ & $\Phi(P_q)$ conjugate to \\
       \hline
     (1)&  $i,i+1 \not\in \ind(q)$ & $ G_0$ \\
      (2) & $(i,i+1)$ is an index pair for $q$ & $N_0$ \\
       (3) & \begin{tabular}{c}  $q= \epsilon_j x_i^2 + x_ix_j +\epsilon_{k}x_{i+1}^2+x_{i+1}x_k +p$ \\ where $\epsilon_{\max(j,k)}=1$ \\
    OR  $q = \epsilon_jx_j^2 + x_jx_i + \epsilon_i x_{k}^2+x_{k}x_{i+1} +p$ \\   where $\epsilon_{\min(j,k)}=1$   \\ (where $i,i+1,j,k \not\in \ind(p)$)\end{tabular}
 & $T_0$ \\
     
     (4)&  $s_iq = q'$ with $q\not= q'$ and $q'$ normal & $U_0$ \end{tabular} \end{center}
  
\end{cor}
\begin{proof} The proof follows from Lemma~\ref{lem-porbit}. For case (2), we notice that if $s_iq=q'$ with $q' \not= q$ and $q'$ normal, then the normal component(s) of $q$ containing $i$ and/or $i+1$ are of the type $\epsilon_jx_j^2 + \delta_jx_jx_i + \epsilon_{k}x_k^2 + \delta_kx_kx_{i+1}$ or  $\epsilon_jx_j^2 + \delta_jx_jx_i + \epsilon_{i+1}x_{i+1}^2 + \delta_{i+1}x_{i+1}x_{k}$ or  $\epsilon_ix_i^2 + \delta_ix_ix_j + \epsilon_{k}x_k^2 + \delta_kx_kx_{i+1}$ with and appropriate conditions on $j,k$ with respect to $i,i+1$ and on the various $\epsilon_*, \delta_* \in \{0,1\}$.
  \end{proof}
  \begin{rem} With the information from the previous Lemma, we can construct the so-called {\bf Brion graph} $G_n$. The vertices are labeled by normal quadratic forms. An edge with label $i$ connects two vertices $q$ and $q'$ if $Bq'  \subset P_iq$. Moreover, the vertex $q$ is placed above $q'$ if $\rank Bq > \rank Bq'$. The various possibilities, given in Lemmas~\ref{lem-porbit} and~\ref{lem-pstab}, are summarized in the following table. Here $P = P_i$.
    \begin{center}
      \begin{tabular}{ccc}
        $\Phi(P_q)$ & $P q$ & Graph \\
   
        \hline
        $G_0$ & $B q$ &  \begin{tikzpicture}
  \tikzset{vertex/.style={draw,circle,inner sep=1.5pt, fill=black}}
  \node[vertex,label = left:$q$] (1) at (0,0)  { }; \end{tikzpicture} \\      \ & \ & \ \\
        $U_0$ & $B q \cup B q'$ & \multirow{2}*{ \begin{tikzpicture}[scale=1]
  \tikzset{vertex/.style={draw,circle,inner sep=1.5pt, fill=black}}
    
    \node[vertex,label=right:{\small $q$}] (1) at (0,0) { };
    \node[vertex,label=right:{\small $q'$}] (2) at (0,0.8) { };
 
    \draw (1) -- (2) node[midway,left] {\small $i$};
\end{tikzpicture}} \\
        & $\rank(q') = \rank(q)-1$ & \\       \ & \ & \ \\
        $T_0$ & $Bq \cup Bq' \cup Bq''$ & \multirow{3}*{\begin{tikzpicture}[scale=1]
  \tikzset{vertex/.style={draw,circle,inner sep=1.5pt, fill=black}}
   
    \node[vertex,label=right:{\small $q$}] (2) at (0.8,0.8) { };
    \node[vertex,label=left:{\small $q'$}] (1) at (0,0) { };
    \node[vertex,label=right:{\small $q''$}] (3) at (1.6,0) { };
 
    \draw (1) -- (2) node[midway,left] {\small $i$};
      \draw (2) -- (3) node[midway,right] {\small $i$};
\end{tikzpicture}  }\\
                    & $\rank(q') = \rank(q'') = \rank(q)-1$ &  \\
         & & \\       \ & \ & \ \\
            $N_0$ & $Bq \cup Bq'$ & \multirow{2}*{ \begin{tikzpicture}
  \tikzset{vertex/.style={draw,circle,inner sep=1.5pt, fill=black}}
\node[vertex,label = right:$q$] (1) at (0,0.8)  { }; 
  \node[vertex,label = right: $q' $] (2) at (0,0)  { };
\draw[thick] (1.110) -- (2.250) node[midway,left] {\small $i$};
\draw[thick] (1.060) -- (2.290);
\end{tikzpicture}}\\
               & $\rank(q) > \rank(q')$ & 
      \end{tabular}
      \end{center}

    \end{rem}
 \medskip
                                                                 \begin{example}
                                                                   For $n = 2$ we have the following Brion graph $G_2$. The edges denote the action of $s_1 = (1\ 2)$.

                                                                   
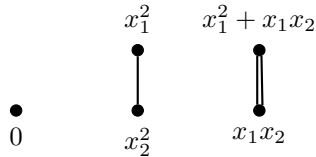
\begin{center}
                                                                     \begin{tikzpicture}[scale = 0.8]
  \tikzset{vertex/.style={draw,circle,inner sep=1.5pt, fill=black}}
  \node[vertex,label = below:{\footnotesize $0$}] (1) at (0,0)  { };
  \node[vertex,label = below: {\footnotesize $x_2^2$}] (2) at (2,0)  { };
  \node[vertex,label = below:{\footnotesize $x_1x_2$}] (3) at (4,0)  { };
  \node[vertex,label = {\footnotesize $x_1^2$}] (4) at (2,1)  { };
  \node[vertex,label = {\footnotesize $x_1^2+x_1x_2$}] (5) at (4,1)  { };

  \draw[thick] (2) to (4);
  \draw[thick] (3.110) to (5.250);
  \draw[thick] (3.060) to (5.290);
\end{tikzpicture}
                                                               \end{center}

                                                               For $n=3$ we have the following Brion graph $G_3$. The action of $s_1 = (1\ 2)$ is denoted by the label $1$, the action of $s_2$ by the label $2 = (2\ 3)$.
                                                                
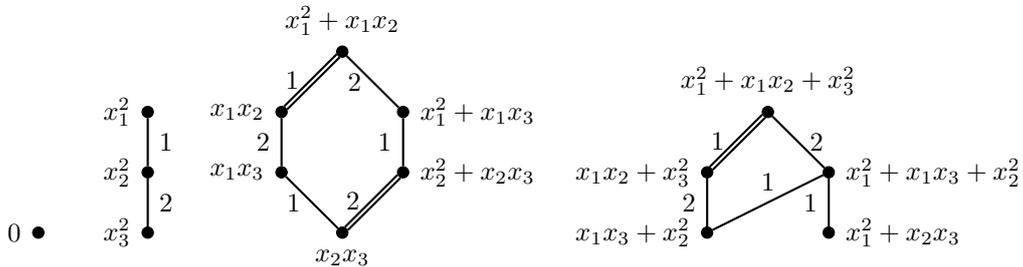
\begin{center}
                                                                 \begin{tikzpicture}[scale=0.8]
  \tikzset{vertex/.style={draw,circle,inner sep=1.5pt, fill=black}}
  \node[vertex,label = left: {\footnotesize $0$}] (1) at (0,0)  { }; 
  \node[vertex,label = left:  {\footnotesize $x_3^2$}] (2) at (1.8,0)  { };
  \node[vertex,label = left: {\footnotesize  $x_2^2$}] (3) at (1.8,1)  { };
   \node[vertex,label = left: {\footnotesize  $x_1^2$}] (4) at (1.8,2)  { };
  \node[vertex,label = below: {\footnotesize $x_2x_3$}] (5) at (5,0)  { };
   \node[vertex,label = left: {\footnotesize $x_1x_3$}] (6) at (4,1)  { };
   \node[vertex,label = right: {\footnotesize $x_2^2+x_2x_3$}] (7) at (6,1)  { };
     \node[vertex,label = left: {\footnotesize $x_1x_2$}] (8) at (4,2)  { };
  \node[vertex,label = right: {\footnotesize $x_1^2+x_1x_3$}] (9) at (6,2)  { };
  \node[vertex,label =  {\footnotesize $x_1^2+x_1x_2$}] (10) at (5,3)  { };
    \node[vertex,label = right: {\footnotesize $x_1^2+x_2x_3$}] (11) at (13,0)  { }; 
    \node[vertex,label = left: {\footnotesize $x_1x_3 +x_2^2$}] (12) at (11,0)  { };
    \node[vertex,label = left: {\footnotesize $x_1x_2 +x_3^2$}] (13) at (11,1)  { };
       \node[vertex,label = above: {\footnotesize $x_1^2+x_1x_2 +x_3^2$}] (14) at (12,2)  { };
       \node[vertex,label = right: {\footnotesize $x_1^2+x_1x_3 +x_2^2$}] (15) at (13,1)  { };

       \draw[thick] (2) -- (3)  node[midway,right] {\footnotesize $2$};
       \draw[thick] (3) -- (4) node[midway,right] {\footnotesize $1$};
       \draw[thick] (5) -- (6) node[midway,left] {\footnotesize $1$};
       \draw[thick] (6) -- (8)  node[midway,left] {\footnotesize $2$};
       \draw[thick] (8.065) -- (10.205)  node[midway,left] {\footnotesize $1$};
       \draw[thick] (8.025) -- (10.245) ;
       \draw[thick] (10) -- (9)  node[midway,left] {\footnotesize $2$};
       \draw[thick] (9) -- (7)  node[midway,left] {\footnotesize $1$};
       \draw[thick] (5.065) -- (7.205)  node[midway,left] {\footnotesize $2$};
       \draw[thick] (5.025) -- (7.245) ;
      
       \draw[thick] (12) -- (13)  node[midway,left] {\footnotesize $2$};
       \draw[thick] (13.065) -- (14.205)   node[midway,left] {\footnotesize $1$};
       \draw[thick] (13.025) -- (14.245) ;
       \draw[thick] (14) -- (15)  node[midway,right] {\footnotesize $2$};
       \draw[thick] (11) -- (15)  node[midway,left] {\footnotesize $1$};
        \draw[thick] (15) -- (12)  node[midway,above] {\footnotesize $1$};

  \end{tikzpicture}
                                                                     \end{center}

\end{example}

\section{ Borel orbit covers and the action of $S_n$}

As already remarked in \cite{knop1995set}, there is generally no natural Weyl group action on the set of $B$-orbits in the characteristic $2$ case. For example, there are five non-degenerate $B$-orbits in $V_3$, namely those with normal representatives $$ x_1^2+x_1x_2 + x_3^2, \ \ x_1^2+x_1x_3 + x_2^2, \ \ x_1x_2 + x_3^2, \ \ x_1x_3 + x_2^2, \ \ x_1^2+x_2x_3.$$
According to the procedure described in \cite{knop1995set} for $\chr p \not= 2$, the simple reflection $s_1$ would fix the Borel orbits associated to the first two normal forms, but $s_2$ interchanges them. In other words, the braid relation $s_1s_2s_1 = s_2s_1s_2$ is not respected. As shown by Knop, instead of considering $B$-orbits, one should consider their double covers, or equivalently, the subgroups of index $\leq 2$ of the isotropy group $B_q$ of the quadratic form $q$.

For the rest of this article we will restrict ourselves to the set $Q_n$ of non-degenerate quadratic normal forms $q$ in $n$ variables which are of maximal $B$-rank $n$. As it turns out, the stabilizer $B_q$ is unipotent.

We compute first $B_q$ for quadratic forms $q$ with a connected diagram.

   \begin{lem}\label{lem-unistab} Let $r= \lfloor \frac{n}{2} \rfloor$ and $q = \sum_{k=1}^rx_{i_k}^2 + x_{i_k}x_{j_k} + \epsilon x_{i_{r+1}} \in Q_n$ be connected. So $\epsilon = 0$ for $n$ even and $\epsilon = 1$ for $n$ odd. Then the stabilizer $B_q$ consists of all elements $u = (u_{k,l}) \in B_q$ which satisfy: \begin{enumerate} \item\label{cond-1} $u_{k,l} = 0$ if $l \not= k =j_s$ for some $s$ or $(k,l)=(i_s,i_t)$ for some $1 \leq s \not= t \leq r$,
     \item\label{cond-2} $u_{i_s, j_t} = u_{i_t, j_s}$ for all $1\leq s \not= t \leq r$,
     \item\label{cond-3} $u_{i_1, j_1} + u_{i_2,j_2} + \cdots + u_{i_r, j_r} \in \begin{cases}  \KK & n \mbox{ odd and } i_{r+1} < 2r+1   \\
  \{0, 1\}    &  \mbox{ otherwise }. \end{cases}$
       \end{enumerate}
   \end{lem}
   \begin{proof}
     Let $u = (u_{k,l}) \in B_q$ be generic. We have $u_{k,l} = \delta_{k,l}$ for $k\geq l$. We first show that $u \in B_q$ satisfies conditions \ref{cond-1}. -- \ref{cond-3}.

     \noindent Notice that since $q$ is normal, of maximal rank, and connected, we have that $i_s< i_{s+1} < j_s < j_{s+1}$ for all $s < n$. Therefore, $u_{i_t,i_s} = u_{j_t,j_s} =0$ for $t>s$ and $u_{j_t,i_s} = 0$ for $t \geq s-1$.
     
     {\bf Proof of 1.:} We use induction on $s$. For $s=1$ we have for $l \not= i_1$: \begin{equation}\label{eq1} \delta_{l j_1} =  \coef(x_{i_1}x_l,q)=\coef(x_{i_1}x_{l}, u  q)  = \sum_{s = 1}^r\underbrace{u_{i_s i_1}}_{\delta_{s 1}} \underbrace{u_{j_s l}}_{\delta_{j_s l}}  +  \sum_{s = 1}^ru_{i_sl} \underbrace{u_{j_si_1}}_{=0}=  u_{j_1 j_t}  \end{equation} where $\delta_{g h}$ is the Kronecker delta. Since $u_{j_1 i_1} = 0$ we have  $u_{j_1 l} = 0$ for all $l\not= j_1$.
     
     For $s=2$ we have for $l \not= i_2$:
     \begin{align}\label{eq2} \delta_{l j_2} &=  \coef(x_{i_2}x_l,q) = \coef(x_{i_2}x_{l}, u  q)   = \sum_{k = 1}^r\underbrace{u_{i_k i_2}}_{=0 \mbox{ {\tiny for} } k>2} u_{j_k l}  +  \sum_{k = 1}^ru_{i_k l} \underbrace{u_{j_k i_2}}_{=0}= \\ &= u_{i_1 i_2} u_{j_1 l} + u_{i_2 i_2}u_{j_2 l}  = \begin{cases}u_{i_1 i_2}  & l=j_1 \\ 1 &  l = j_2 \\ u_{j_2 l} & l \not= i_2, j_1,j_2 \end{cases}. \end{align} So $u_{i_1 i_2} = 0$ and, since $u_{j_2 j_1} = u_{j_2 i_2} = 0$, we have $u_{j_2 l} = 0$ for all $l\not= j_2$.

     We assume that our induction hypothesis holds for up to $s-1$. We have for $l \not= i_s$:
       \begin{align}\label{eq3} \delta_{l j_s} &=  \coef(x_{i_s}x_l,q) = \coef(x_{i_s}x_{l}, u  q) = \sum_{k = 1}^r\underbrace{u_{i_k i_s}}_{(=0 \mbox{ {\tiny for} } k > s)\ } \underbrace{u_{j_k l}}_{\ (\delta_{j_k l} \mbox{ {\tiny for }} k < s)}  +  \sum_{k = 1}^ru_{i_k l} \underbrace{u_{j_k i_s}}_{=0}= \\ & = \sum_{k=1}^su_{i_k i_s} \delta_{j_k l}   = \begin{cases}u_{i_k i_s}   & l=j_1, \ldots, j_{s-1} \mbox{ (induc.hyp.)}, \\ 1 &  l = j_s, \\ u_{j_s l} & l \not= j_1,\ldots, j_s, i_s \end{cases}. \end{align} 
   So $u_{i_1 i_s}= \cdots = u_{i_{s-1} i_s} = 0$ and, since $u_{j_s j_1} = \cdots = u_{j_s j_{s-1}} = u_{j_s i_s} = 0$, then $u_{j_s l} = 0$ for all $l \not= j_s$. Therefore by induction Condition \ref{cond-1}. holds. 
                   
                       \noindent
                     {\bf  Proof of \ref{cond-2}.:} This follows from
                       \begin{align}\label{eq4} 0 = \coef(x_{j_s}x_{j_t}, u  q) & = \sum_{k}u_{i_k j_s}\underbrace{u_{j_k j_t}}_{ = 0 \mbox{ for } k \not= t} + \sum_{k}u_{i_k j_t}\underbrace{u_{j_k j_s} }_{ = 0 \mbox{ for } k \not=s} \\ & = u_{i_t j_s} u_{j_t j_t} + u_{i_s j_t}u_{j_s j_s} =  u_{i_t j_s} + u_{i_s j_t} \end{align}

                       \noindent
                       {\bf Proof of \ref{cond-3}.:}  We consider $ 0 = \sum_{t=1}^r\coef(x_{j_t}^2, q) = $ \begin{align}\label{eq5} = \sum_{t=1}^r\coef(x_{j_t}^2, u   q)& = \sum_{s=1}^r \left( \sum_{i_s \leq j_t} u_{i_s j_t}^2 + \sum_{s \leq t}u_{i_s j_t} \underbrace{u_{j_s j_t}}_{ = 0 \mbox{ for } s \not= t}  \right) + \epsilon\left( \sum_{t=1}^r u_{i_{r+1}j_t}^2\right) \\
                         & = \sum_{t=1}^r \left( \sum_{i_s \leq j_t} u_{i_s j_t}^2 + u_{i_t j_t}\right)  + \epsilon\left( \sum_{t=1}^r u_{i_{r+1} j_t}^2\right) \end{align}
                       Since $q$ is of maximal $B$-rank and normal then we have two possibilities for the relative positioning of $(i_t,j_t)$ and $(i_s, j_s)$ for $t \not= s$, $s,t \not= r+1$. Either $i_t < i_s < j_t < j_s$ or $i_s < i_t < j_s < j_t$. Either way, $i_s < j_t$ if and only if $i_t < j_s$. So both $u_{i_s j_t}$ and $u_{i_t j_s}$ occur in the sum on the left, cancelling each other out. In other words, we have
                       $$0 = \sum_{t=1}^r (u_{i_t j_t}^2 + u_{i_t j_t})   + \epsilon\left( \sum_{t=1}^r u_{i_{r+1}j_t}^2\right) .$$ When $n$ is even, $\epsilon = 0$. When $i_{r+1} = n$ then all $j_t < i_{r+1}$ so the $ u_{i_{r+1}j_t}=0$ for all $t$. In these cases our equation reduces to $$ 0 = \left(\sum_{t=1}^ru_{i_t j_t} \right)^2 + \sum_{t=1}^ru_{i_t j_t}.$$ So $\sum_{t=1}^ru_{i_t j_t}$ is a root of the polynomial $x^2 + x$, proving the claim.

                       When $n$ is odd and $i_{r+1} < n$ then $\epsilon = 1$ and the $u_{i_{r+1},j_t} \in \KK$ for all $t$.  The equation is therefore
                       $$  \left(\sum_{t=1}^ru_{i_t j_t} \right)^2 + \sum_{t=1}^ru_{i_t j_t} = \left( \sum_{t=1}^r u_{i_{r+1}j_t}\right)^2,$$ proving our claim.

                       That an element $u = (u_{k,l})$ satisfying conditions \ref{cond-1}. - \ref{cond-3} also satisfies $u \in B_q$ follows from equations \ref{eq1} - \ref{eq5}.
                     \end{proof}
                      \begin{lem}\label{lem-disconnect} Let $q = q_1 + q_2 \in Q_n$ be such that $\max(\ind(q_1)) < \min(\ind(q_2))$. Then for every element $u \in B_q $ we have $u_{g,h} = 0$ for every $(g,h) \in \ind(q_1) \times \ind(q_2)$.  In particular, $B_q \cong B_{q_1} \times B_{q_2}$.

 \end{lem}
 \begin{proof} Let $q_1$ be connected with index pairs $\{(i_k, j_k) | 1 \leq k \leq t \}$ for which $i_1 < \cdots < i_t$. Let $g = j_s \in \ind(q_1)$ and $h \in \ind(q_2)$. Since $q_1$ is connected then by Lemma~\ref{lem-unistab} we have \begin{equation}\label{eqn-dis1} 0  = \coef(x_{i_s}x_h, q) =\coef(x_{i_s}x_h, u q) = \sum_k \underbrace{u_{i_k i_s}}_{\delta_{k s}}u_{j_k h} +  \sum_k u_{i_k h}\underbrace{u_{j_k i_s}}_{=0} = u_{j_s h}. \end{equation}
   In the left-hand above sum in (\ref{eqn-dis1}), the term $u_{i_k i_s} = 0$ by Lemma~\ref{lem-unistab} for $k < s$ and by $i_k > i_s$ for $k>s$. The analogue holds for the term $u_{j_k i_s}$ in the right-hand sum of  (\ref{eqn-dis1}).

   For $g = i_s \in \ind(q_1)$ we have
   \begin{equation}\label{eqn-dis2} 0 = \coef(x_{j_s}x_h, q) =
     \coef(x_{j_s}x_h, u q) = \sum_ku_{i_k j_s}u_{j_k h} +  \sum_k u_{i_k h}u_{j_k j_s} = u_{i_s h} \end{equation}
   In the left-hand above sum in (\ref{eqn-dis2}), the term $u_{j_k h} = 0$ for $k \leq s$ by  (\ref{eqn-dis1}). For $k > s$ we have $i_k > j_s$ so the term $u_{i_k j_s} = 0$. The left-hand sum thereby equals $0$. The term $u_{j_k j_s}$ from the right-hand sum equals $0$ for all $k \not= s$, either due to Lemma~\ref{lem-unistab} or due to $j_k > j_s$. We have therefore $B_q \cong B_{q_1} \times B_{q_2}$.

   Using the same argument for $q_2 = q_2'+ q_3$ with $q_2'$ connected, we obtain $B_{q_2} \cong B_{q_2'} \times B_{q_3}$.  Repeating this argument, we see that the claim holds for all of $q$.
   
 \end{proof}

 \begin{rem} Let $q \in Q_n$.
   Using the notation from Lemma \ref{lem-unistab}, we consider the map $\phi:B_q  \rightarrow \BZ_2$, \begin{equation}  u \in B_q  \mapsto  \epsilon_u := \sum_{k=1}^ru_{i_k j_k} \in \{0,1\}. \end{equation} This map is a group homomorphism. Indeed, let $u, v \in B_q$ and $w:= uv$. Then
   \begin{equation} e_w = \sum_{k=1}^rw_{i_k j_k} = \sum_{k=1}^r \sum_{l=1}^ru_{i_k l}v_{l j_k}=  \sum_{k=1}^r (u_{i_kj_k} + v_{i_k j_k}) = \epsilon_u + \epsilon_v . \end{equation}
   The kernel of this homomorphism is the connected subgroup $B_q^0 := \phi^{-1}(0)$ of $B_q$. It follows $B_q/B_q^0 \cong \BZ_2$. For $u \in B_q$ with $\phi(u)=1$ we have that $B_q^0q = uB_q^0q= B_qq$ which leads us to the following definition.
   
 \end{rem}
 \begin{defn} Let $q \in Q_n$. A {\bf double cover} of $B_q$ is any subgroup $H$ of $B_q$ with $B_q^0 \subset H \subset B_q$ and $[B_q: H] \leq 2$ where $B_q^0$ is the connected component of the identity of $B_q$. The {\bf group of components } of $B_q$ is $\pi_0(B_q):= B_q/B_q^0$.

 \end{defn}
 \begin{rem} The trivial double cover $B/B_q \times \{0,1\}$ corresponds to the subgroup $H=B_q$ of index $1$.  The other double covers will be called {\bf proper}.

   \end{rem}
                      \begin{example} For $G = GL(3,\KK)$ we have $Q_3 = \{ x_1^2 + x_{1}x_{ 2} +  x_3^2, x_1^2 + x_{1}x_{ 3} +x_2^2\}$. The unipotent stabilizers are as follows:

     \begin{center}
     \begin{tabular}{cc}
       $q \in Q_3$ &   $B_{q}$ \\
       \hline
       \begin{tabular}{c} $ x_1^2 + x_{1}x_{ 2} +  x_3^2 $ \\ \raisebox{-.5\height}{\begin{tikzpicture}[scale=0.8]
  \tikzset{vertex/.style={draw,circle,inner sep=1pt, fill=black}}
   \tikzset{overtex/.style={draw,circle,inner sep=1pt}}
    
    \node[vertex,label=below:{\tiny $1$}] (1) at (0,0) { };
    \node[overtex,label=below:{\tiny $2$}] (2) at (1,0) { };
    \node[vertex,label=below:{\tiny $3$}] (3) at (2,0) { };

    \draw (1) to [out=30,in=150] (2);
   
  \end{tikzpicture}}\end{tabular}  & $\left\{ \begin{pmatrix} 1 & u_{1,2} & 0  \\ 0 & 1 & 0  \\ 0 & 0 & 1  \end{pmatrix} \ \ \mid \ \ u_{1,2} \in \{ 0,1 \} \ \right\}$ \\  & \\
  \begin{tabular}{c} $ x_1^2 + x_{1}x_{ 3} +x_2^2$ \\ \raisebox{-.5\height}{ \begin{tikzpicture}[scale=0.8]
  \tikzset{vertex/.style={draw,circle,inner sep=1pt, fill=black}}
   \tikzset{overtex/.style={draw,circle,inner sep=1pt}}
    
    \node[vertex,label=below:{\tiny $1$}] (1) at (0,0) { };
    \node[vertex,label=below:{\tiny $2$}] (2) at (1,0) { };
    \node[overtex,label=below:{\tiny $3$}] (3) at (2,0) { };
    \draw (1) to [out=30,in=150] (3);
\end{tikzpicture}}\end{tabular}    & $ \left\{ \begin{pmatrix} 1 & 0 & u_{1,3} \\ 0 & 1 &0 \\ 0 & 0 & 1  \end{pmatrix} \ \ \mid \   \ u_{1,3} \in \KK  \right\} $ \\   &  \\     \end{tabular}\end{center}
Therefore the orbit $B  ( x_1^2 + x_{1}x_{ 2} +x_3^2)$ has one proper double cover while $B   ( x_1^2 + x_{1}x_{ 3} +  x_2^2)$ has none.
     \end{example} 
   \begin{example} For $G = GL(4,\KK)$ we have $Q_4 = \{x_1^2 + x_{1}x_{ 2} +  x_3^2+ x_{3}x_{ 4},x_1^2 + x_{1 }x_{3} +x_2^2+ x_{2 }x_{4}\}$. The unipotent stabilizers are as follows:

     \begin{center}
     \begin{tabular}{cc}
       $q \in Q_4$ &   $B_{q}$ \\
       \hline
       \begin{tabular}{c} $ x_1^2 + x_{1}x_{ 2} +  x_3^2+ x_{3}x_{ 4} $ \\ \raisebox{-.5\height}{\begin{tikzpicture}[scale=0.8]
  \tikzset{vertex/.style={draw,circle,inner sep=1pt, fill=black}}
   \tikzset{overtex/.style={draw,circle,inner sep=1pt}}
    
    \node[vertex,label=below:{\tiny $1$}] (1) at (0,0) { };
    \node[overtex,label=below:{\tiny $2$}] (2) at (1,0) { };
    \node[vertex,label=below:{\tiny $3$}] (3) at (2,0) { };
    \node[overtex,label=below:{\tiny $4$}] (4) at (3,0) { };

    \draw (1) to [out=30,in=150] (2);
    \draw (3) to [out=30,in=150] (4);
\end{tikzpicture}} \end{tabular}  & $\left\{ \begin{pmatrix} 1 & u_{1,2} & 0 & 0 \\ 0 & 1 & 0 & 0 \\ 0 & 0 & 1 & u_{3,4} \\ 0 & 0 & 0 & 1 \end{pmatrix} \ \ \mid \ \ u_{1,2}, u_{3,4} \in \{ 0,1 \} \ \right\} $ \\   & \\
  \begin{tabular}{c} $ x_1^2 + x_{1}x_{ 3} +x_2^2+ x_{2}x_{ 4}$ \\ \raisebox{-.5\height}{\begin{tikzpicture}[scale=0.8]
  \tikzset{vertex/.style={draw,circle,inner sep=1pt, fill=black}}
   \tikzset{overtex/.style={draw,circle,inner sep=1pt}}
    
    \node[vertex,label=below:{\tiny $1$}] (1) at (0,0) { };
    \node[vertex,label=below:{\tiny $2$}] (2) at (1,0) { };
    \node[overtex,label=below:{\tiny $3$}] (3) at (2,0) { };
    \node[overtex,label=below:{\tiny $4$}] (4) at (3,0) { };

    \draw (1) to [out=30,in=150] (3);
    \draw (2) to [out=30,in=150] (4);
\end{tikzpicture}}\end{tabular}    & $ \left\{ \begin{pmatrix} 1 & 0 & u_{1,3} & a \\ 0 & 1 & a & u_{1,4} \\ 0 & 0 & 1 & 0 \\ 0 & 0 & 0 & 1 \end{pmatrix} \ \ \mid \  a\in\KK, \ u_{1,3}+u_{2,4} \in \{ 0,1\}  \right\} $ \\   &  \\     \end{tabular}\end{center}
Therefore the orbit $B  ( x_1^2 + x_{1}x_{ 3} +x_2^2+ x_{2}x_{ 4})$ has two double covers and the orbit $B   ( x_1^2 + x_{1}x_{ 2} +  x_3^2+ x_{3 }x_{4})$ has four covers.
\end{example}
\noindent
A direct consequence of Lemmas~\ref{lem-disconnect} and \ref{lem-unistab} is: 
     \begin{cor}\label{lem-numcomp} Let $q \in Q_n$. Let $f$ be the number of connected components in $q$ for $n$ even  or in $\tilde{q}$ for $n$ odd. Then the$$\mbox{ number of connected components in } B_q = \begin{cases} 2^f & n \mbox{ even} \\ 2^{f-1} & n \mbox{ odd} \end{cases}.$$
The group of components $\pi_0(B_q) \cong \BZ_2^f$. 
   \end{cor}
  
   \begin{lem} Let $K_n$ be the set of double covers for all $q \in Q_n$. Let $M_n$ be the set of subsets of $\{1, \ldots, n\}$ with $\lfloor \frac{n}{2} \rfloor$ elements. Let
     $$Z_n:= \{ (q, \epsilon) \ | \ q \in Q_n,\ \epsilon \in \BF_2^{cc(f)}\ \}.$$ Then there are natural bijections
     $$K_n \xrightarrow{\rho} Z_n \xrightarrow{\pi} M_n.$$
   \end{lem}
   \begin{proof} {\bf (Description of $\rho$)}
     Let ${\mathcal U} \in K_n$. Then ${\mathcal U} \subset B_q$ for a certain $q \in Q_n$. We describe $q = q^{(1)} + \cdots + q^{(f)}$ as a sum of its connected components $ q^{(1)}, \ldots ,  q^{(f)}$. Each such component $q^{(l)}$ is associated to a set of index pairs $\{(i_k^{(l)},j_k^{(l)}) \ | \ 1\leq k \leq d_l\}$. According to Lemma \ref{lem-unistab}, the subset ${\mathcal U}$ is uniquely determined by the relations
     $$\sum_ku_{i_k^{(l)},j_k^{(l)}} = \epsilon_l \ \ \ \mbox{ for } \epsilon_l \in \BF_2, \ \  1 \leq l \leq cc(f).$$
     We set $\epsilon = (\epsilon_1, \ldots, \epsilon_{cc(f)})$. The assignment ${\mathcal U} \leftrightarrow (q,\epsilon)$ is therefore clear.

     {\bf (Description of $\pi$)}
   For simplicity's sake, we first discuss how this works for $n = 2r$ even.
Let $(q, \epsilon) \in Z_n$ with $q = q^{(1)} + \cdots + q^{(f)}$ as a sum of its connected components $ q^{(1)}, \ldots ,  q^{(f)}$. Let $(i^{(l)}_1, j^{(l)}_1), \ldots, (i_{d_l}^{(l)},j_{d_l}^{(l)})$ be the index pairs for the connected component $q^{(l)}$ of $q$. Then we have
     $$ {\bf m}^{l} := \begin{cases} \{ i^{(l)}_1, \dots, i^{(l)}_{d_l} \} & \mbox{ if } \epsilon_l = 0  \\
       \{ j^{(l)}_1, \dots, j^{(l)}_{d_l} \}  & \mbox{ if } \epsilon_l = 1  \end{cases}$$
  Then ${\bf m} := \cup_{t = 1}^f{\bf m}^{t} \in M_n$.

     Conversely, every subset of $\{1, \ldots, n\} \in M_n$ of $n/2$ elements can be associated to a unique element $(q,\epsilon) \in Z_n$. Let ${\bf m} = \{ m_1, \ldots, m_r \}$ be such a subset where $1 \leq m_1 < \cdots < m_r \leq n $. With the complement subset $\overline{{\bf m}} := \{1, \ldots, n\} \backslash {\bf m}$ we have the following relation: if $m_a < \overline{m}_a$ then $(m_a, \overline{m}_a)$ is an index pair. Otherwise $(\overline{m}_a,m_a)$ is an index pair. The set of all index pairs determine a quadratic form $q =  q^{(1)} + \cdots + q^{(f)}$, where the $q^{(t)}$ are the various connected components. Then $\bm^{(t)} \subset {\bf m}$ is the set of indices corresponding to $q^{(t)}$. This subsequence $\bm^{(t)}$ corresponds to  the relation $$ \sum_{s } u_{\bm^{(t)}_s,\overline{\bm}^{(t)}_s}  = \begin{cases} 0 & \text{ if } \bm^{(t)}_{i}>\overline{\bm}^{(t)}_i \text{ for all }  i  \\
       1 & \text{ if } \bm^{(t)}_{i}<\overline{\bm}^{(t)}_i \text{ for all }  i  \end{cases} =: \epsilon_t.$$

    We now consider this mapping for $n = 2r+1$ odd. Using the same notation, we have that $\ind(q^{(f)})$ contains $i_{r+1}$ and therefore is associated to the condition  $$\mbox{ if } d_f>1 \mbox{ then } \sum_{a = 1}^{d_f} u_{i_a^{(l)},j_a^{(l)}} \in \KK $$ for the unipotent stabilizer. In particular, there is only one connected component coming from this relation and the associated subset  $m^{(f)} = \{ i_1^{(f)}, \ldots, i_{d_f}^{(f)} = i_{r+1}\}$.  There would be $2^{f-1}$ distinct subsets $m_{\epsilon}$ in which the $\epsilon_1, \ldots, \epsilon_{f-1} \in \{ 0,1\} $ vary and $\epsilon_f = 1$ is fixed.

     Conversely, given a sequence $m \subset \{1,\ldots, 2r+1\}$ with $r+1$ elements, we set $\overline{m} = \{1, \ldots, 2r+2\} \backslash m$ and the index pairs are determined as for $n$ even. The associated quadratic form $q$ is obtained by setting $x_{2r+2} = 0$. 
   \end{proof}
   \begin{cor}
 $$ |K_n| = {n \choose \lfloor \frac{n}{2} \rfloor}.$$
     \end{cor}
   \begin{example} Let $m = \{1,3,4, 9,10\} \in M_{10}$. Then $\overline{m} = \{2,5,6,7,8 \}$. The index pairs of the associated quadratic form $q \in Q_{10}$ are $(i_1,j_1) = (1,2)$, $(i_2,j_2) = (3,5)$, $(i_3,j_3)=(4,6)$, $(i_4,j_4) = ( 7,9)$, and $(i_5,j_5) = (8,10)$ and its diagram is: 

     
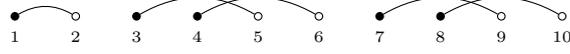
\[
        \begin{tikzpicture}[scale=0.8]
    \tikzset{vertex/.style={draw,circle,inner sep=1pt, fill=black}}
    \tikzset{overtex/.style={draw,circle,inner sep=1pt}}
    
    \node[vertex,label=below:{\tiny $1$}] (1) at (0,0) { };
    \node[overtex,label=below:{\tiny $2$}] (2) at (1,0) { };
    \node[vertex,label=below:{\tiny $3$}] (3) at (2,0) { };
    \node[vertex,label=below:{\tiny $4$}] (4) at (3,0) { };
     \node[overtex,label=below:{\tiny $5$}] (5) at (4,0) { };
    \node[overtex,label=below:{\tiny $6$}] (6) at (5,0) { };
    \node[vertex,label=below:{\tiny $7$}] (7) at (6,0) { };
    \node[vertex,label=below:{\tiny $8$}] (8) at (7,0) { };
     \node[overtex,label=below:{\tiny $9$}] (9) at (8,0) { };
    \node[overtex,label=below:{\tiny $10$}] (10) at (9,0) { };

    \draw (1) to [out=30,in=150] (2);
    \draw (3) to [out=30,in=150] (5);
    \draw (4) to [out=30,in=150] (6);
    \draw (7) to [out=30,in=150] (9);
     \draw (8) to [out=30,in=150] (10);
  \end{tikzpicture}
\] We see that $m = \{i_1^{(1)}, i_1^{(2)},i_2^{(2)}, j_1^{(3)}, j_2^{(3)}\}$, which is associated to the connected component of $B_q$ with conditions $$ u_{1,2} = 0, \ \ \ u_{3,5} + u_{4,6} = 0, \ \ \ u_{7, 9} + u_{8,10} = 1 .$$ The associated element in $Z_{10}$ is $(q, \epsilon)$ where $$ q = x_1^2+x_1x_2 + x_3^2+ x_3x_5 + x_4^2 + x_4x_6 + x_7^2 + x_7x_9 + x_8^2 + x_8x_{10} \ \ \mbox{ and } \ \ \epsilon = (0,0,1).$$
   \end{example}
 
     \begin{cor}\label{cor-act} The Weyl group $S_n $ acts on the set $K_n$ of connected components of the unipotent stabilizers $B_q$ for $q \in Q_n$.

   \end{cor}
   \begin{proof} The obvious action of $w \in S_n$ on the $\bm \in M_n$ is $w  \bm = (w(m_1), \ldots, w(m_{\lfloor \frac{n}{2} \rfloor}))$. We extend this action to $K_n$ by setting $w\cdot {\mathcal U} := w(\rho({\mathcal U}))$.  We note that, in the case that $n$ is odd, the $S_n$-action never switches $m^{(f)} = \{ i_1^{(f)}, \ldots, i_{d_f}^{(f)} = i_{r+1}\}$ to $\overline{\bm}^{(f)} = \{ j_1^{(f)}, \ldots, j_{d_f}^{(f)} = n+1\}$ because $S_n$ fixes the number $n+1$. In other words, this action fixes the single cover of $B_q$ associated to $\bm^{(f)}$.
   \end{proof}
  
   \begin{example} Let $G = SL(3,\KK)$. There are two non-degenerate quadratic forms of $B$-rank 2, namely   $q:= x_1^2 + x_{1}x_{ 2} +x_3^2$ and $p:= x_1^2 + x_{1}x_{ 3} +x_2^2$. The stabilizer $B_q$ has two components, each determined by a choice of $u_{1,2} \in \{ 0,1\}$.  The $2$-element sets associated to each component are $\{1,3\}, \{2,3\}$.  The stabilizer $U_p$ has one component, associated to the $2$-element set $\{1,2\}$. The Weyl group $S_3$ action is:
     
    
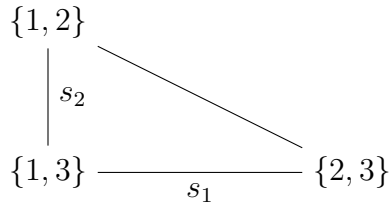
\begin{center}
 \begin{tikzpicture}
                                                                          \node (1) at (0,0) { $\{1,3\}$ };
                                                                     \node (2) at (0,2) { $\{1,2\}$ };
                                                                     \node (3) at (4,0) { $\{2,3\}$ };
                                                                      \draw (1) -- node[right]{$s_2$}  (2);
                                                                       \draw (1) -- node[below]{$s_1$}  (3);
                                                                       \draw (2) -- node[below]{ }  (3);
                                                                       \end{tikzpicture} 
      \end{center} 

       \end{example}
   \begin{example} Let $G = SL(4,\KK)$. There are two non-degenerate quadratic forms of $B$-rank 3, namely   $q:= x_1^2 + x_{1 }x_{2} +x_3^2+ x_{3}x_{ 4}$ and $p:= x_1^2 + x_{1}x_{ 3} +x_2^2+ x_{2}x_{ 4}$. The stabilizer $B_q$ has four components, each determined by a choice of $u_{1,2}, u_{3,4} \in \{0,1\}$. The stabilizer $U_p$ has two components, each determined by $u_{1,3}+u_{2,4} \in \{ 0,1\}$. The $2$-element sets associated to each component is as follows:

     \begin{center}
     \begin{tabular}{cccc}
       $q= x_1^2 + x_{1}x_{ 2} +x_3^2+ x_{3}x_{ 4}$ & &  $p= x_1^2 + x_{1 }x_{3} +x_2^2+ x_{2}x_{ 4}$ & \\
       $u_{1,2}=0, u_{3,4} = 0$ & $\{1,3\}$ & $u_{1,3}+u_{2,4} = 0 $ & $\{1,2\}$ \\
        $u_{1,2}=0, u_{3,4} = 1$ & $\{1,4\}$ & $u_{1,3}+u_{2,4} = 1 $ & $\{3,4\}$ \\
       $u_{1,2}=1, u_{3,4} = 0$ & $\{2,3\}$ &  & \\
          $u_{1,2}=1, u_{3,4} = 1$ & $\{2,4\}$ &  & \end{tabular} \end{center}
       
      The Weyl group action is :

    
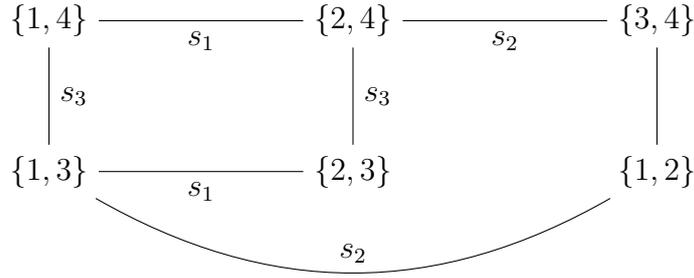
\begin{center}
 \begin{tikzpicture}
                                                                          \node (1) at (0,0) { $\{1,3\}$ };
                                                                     \node (2) at (0,2) { $\{1,4\}$ };
                                                                     \node (3) at (4,0) { $\{2,3\}$ };
                                                                     \node (4) at (4,2) { $\{2,4\}$ };
                                                                     \node (5) at (8,0) { $\{1,2\}$ };
                                                                      \node (6) at (8,2) { $\{3,4\}$ };
                                                                      \draw (1) -- node[right]{$s_3$}  (2);
                                                                       \draw (1) -- node[below]{$s_1$}  (3);
                                                                     \draw (2) -- node[below]{$s_1$}  (4);
                                                                     \draw (4) -- node[right]{$s_3$}  (3);
                                                                     \draw (4) -- node[below]{$s_2$}  (6);
                                                                     \draw (1) to[bend right, edge label =$s_2$]    (5);
                                                                        \draw (5) --  node[right]{ }  (6);
\end{tikzpicture} 
\end{center} \end{example}
\begin{rem} One might wonder when the action of $S_n$ on $M_n$ mirrors that of $S_n$ on the set of all Borel orbits. Let $q \in Q_n$.  It is not hard to show that $s_iB  q = B  (s_iq)$ if $(s_iq)$ is normal. We look at an example when this is not the case: let $G = GL(2,\KK)$ and $q = x_1^2 + x_1x_2$. Then $s_1(q) = x_2^2 + x_1x_2$ which is neither normal nor of maximal rank. In particular $s_1B  q = B  (x_1x_2)$. There is, however, no problem at the level of $M_n$ because $s_1 \{1\} = \{2 \}$ so $s_1$ simply switches one orbit cover of $B  q$ for the other.

\end{rem}

  \begin{lem} The action of $S_n$ on $K_n$ given in Corollary~\ref{cor-act} is the same as the one described in Lemma 5.4 in \cite{knop1995set}.

  \end{lem}
  \begin{proof} It suffices to show that the action of a simple reflection $s_i$ is the same in both Corollary~\ref{cor-act} and Lemma 5.4, \cite{knop1995set}. We compare it with the action given by the table from Lemma 5.4.

     \medskip\noindent
     {\bf Case $\bm \not= s_i\bm $ with $\bm_l = i$ and $\overline{\bm}_l = i+1$ (or vice versa): } The associated quadratic form $q$ has a summand of the form $x_i^2 + x_ix_{i+1}$.  The reflection maps the double cover for $Bq$ corresponding to the condition $u_{i, i+1} = 0$  to the double cover corresponding to $u_{i,i+1} = 1$. Notice that the tag coefficient $\epsilon_l$ changes from $0$ to $1$ (with everything else the same). In this case we have $\Phi(P_q) = N_0$. This same action is denoted by  $s_i: [x_1,\rho] \rightarrow [x_1,\epsilon\rho]$ in Lemma 5.4 in \cite{knop1995set}.
     
      \medskip \noindent
  {\bf Case $s_i \bm = \bm $: }  Then both $i, i+1$ are in $ \bm$ (or $\overline{\bm}$). So $s_i$ maps the corresponding double cover for $q$ onto itself. The quadratic form $q$ has summands of the form $$x_i^2 + x_ix_j + x^2_{i+1} + x_{i+1}x_k \ \ \ \mbox{ where } i+1 < j < k.$$ So $s_iq$ would have nested intervals, making it non normal. In this case $\Phi(P_q) = T_0$. The case $i,i+1 \in \overline{\bm}$ is analogous. This same action is denoted by $s_i: [x_0,\rho] \rightarrow [x_0,\rho]$ in Lemma 5.4 in \cite{knop1995set}.

  \medskip \noindent
  {\bf Case $s_i\bm \not= \bm$ and if $\bm_l = i$ then $\overline{\bm}_l \not= i+1$ (or vice versa):} In this case $s_i$ maps an orbit cover for $q$ to an orbit cover for $q'=s_iq $. This case occurs when $\bm_l = i$ and $\overline{\bm}_k = i+1$ (or vice versa) with $k \not= l$. The corresponding connected components of $q$ are of the form
  $$ x_j^2 + x_jx_i + x_{i+1}^2 + x_{i+1}x_k \ \ \ \mbox{ or } \ \ \ x_i^2 + x_ix_j + x_k^2 + x_kx_{i+1} \ \ \ \mbox{ for some } j,k.$$In this case $\Phi(P_q) = U_0$. This same action is denoted by  $s_i: [x_0,\rho] \rightarrow [x_{\infty},\rho]$  in Lemma 5.4 in \cite{knop1995set}.
 \end{proof}
  
 \begin{rem} One obvious avenue for further study would be the extension of the results from Section 4 to all normal quadratic forms $q$ (and not just the non-degenerate ones of maximal rank). The equations defining the stabilizer $B_q$ tend to be more complicated. The number of covers for $Bq$ equals the number of connected components in $q$ which contain a pure power $x_i^2$. In general, the results and proofs are much more technical and will therefore be presented in a future paper.

   \end{rem}
 
   \bibliographystyle{plain}
 \bibliography{quadbib}

   \end{document}